\documentclass[10pt]{amsart}
\usepackage{amsmath}
  \usepackage{graphics} 
\usepackage{graphicx}  
  \usepackage{epsfig} 
 \usepackage[colorlinks=true]{hyperref}
  \usepackage{float}
  \usepackage{caption}
  \usepackage[english]{babel}
\usepackage{amsmath}
\usepackage{wrapfig}
\usepackage{graphicx}
\usepackage{subcaption}
\usepackage{stackengine}
\usepackage{comment}
\newcommand\dhookrightarrow{\mathrel{%
  \ensurestackMath{\stackanchor[.1ex]{\hookrightarrow}{\hookrightarrow}}
}}
\usepackage{lscape}
\usepackage{amssymb}

\usepackage{amsthm}
\theoremstyle{plain}

\newtheorem*{theorem*}{Theorem}
\newtheorem{theorem}{Theorem}[section]
\newtheorem{corollary}{Corollary}

\newtheorem{lemma}[theorem]{Lemma}

\theoremstyle{definition}
\newtheorem{definition}[theorem]{Definition}
\newtheorem{remark}{Remark}

\title[PDE FTE competition ] 
    {The effect of ``very fast" dispersal on two species competition with drift}


\begin{document}
\maketitle

\centerline{Rana D. Parshad$^{1}$, Erin Ellefsen$^{2}$ and Vaibhava Srivastava$^{1}$  }
\medskip
{\footnotesize

   \medskip
   
    \centerline{1) Department of Mathematics}
 \centerline{Iowa State University}
   \centerline{Ames, IA 50011, USA.}

 \medskip
\centerline{2) Department of Mathematics, Statistics and Computer Science}
 \centerline{St. Olaf College}
   \centerline{Northfield, MN  55057, USA}

 }

\begin{abstract}

Classical theory predicts that for two competing populations subject to a constant downstream drift, the faster disperser will competitively exclude the slower disperser. In the current work, we consider a novel model of a ``much faster" dispersing  species, modeled via a $p$-laplacian operator, competing with a slower disperser. We prove global existence of weak solutions to this model for any positive initial condition, in the regime $\frac{3}{2} < p <2$. Counterintuitively, we  show that while the faster disperser always wins - the ``much faster" disperser could actually lose, for certain initial data. Several numerical simulations are conducted to confirm our analytical findings. Our results have implications for biodiversity, refuge design, and improved biological control, driven by habitat fragmentation and climate change.

%

\end{abstract}

\section{Introduction}

Competition pervades many different strata of the natural world. These include large- and small-scale ecological systems, social systems, political systems and uprisings, military invasions, biological invasions, economic and financial applications, epidemics, virology, and even linguistic systems \cite{porter2008competition, vickers1995concepts, burt2003social}. For these reasons, the classical two-species competition model has been intensively investigated over the years \cite{Cantrell2003, Chen2020, DeAngelis2016, He2013b, He2016a, He2016b, Hutson2003, Lou2006a, Lou2006b, Li2019, Lam2012, Nagahara2018, Nagahara2020}. Species continuously disperse in spatial domains in search of food, mates, resources, and even refuge \cite{zhou2016qualitative}, leading to these dispersal patterns shaping many ecological communities. Thus, of keen interest is the case where there is a resource function that is spatially dependent \cite{He2019, He2013a}. In particular, studying competition systems 
modeled with a spatially explicit reaction-diffusion framework is very relevant and helps predict equilibrium states and their stability, spatial patterns, traveling or invasion waves, all of which help us discern between possible coexistence, exclusion, or extinction of species.  

 Another important mechanism of species movement is drift. This is mostly seen in aquatic systems, such as river or stream flows, where the flow of water will determine the direction of movement of species in the water body, typically from an upstream location to a downstream location - with many works focused on competition in the presence of drift \cite{zhou2016qualitative, cantrell2007advection, zhou2016lotka}. Recent work has also focused on the drift in river or stream communities with a network structure and dispersal in eco-epidemic systems \cite{chen2023evolution, chen2024evolution, salako2024degenerate}. Drift between crop fields could increase in the foreseeable future due to climatic changes, with precipitation projected to increase in the North Central United States over the next several years \cite{lee2023ipcc}. Thus, in these settings, landscape design from a pest or invasive species control point of view becomes increasingly important \cite{haan2020predicting}. Additionally, human activity continues to fragment natural habitat, increasing the importance of understanding species dispersal and the ensuing effects on competition in fragmented habitats \cite{rohwader2022foraging, fletcher2018habitat}.
Due to these environmental changes, species will have to adapt their dispersal patterns. Otherwise, they will inadvertently be driven to extinction. This could also be a telling factor for invasive species control - as climate conditions turn favorable for an invader and unfavorable for the resident, the invader might be able to fill a vacant ecological niche \cite{lekevivcius2009vacant}, that might otherwise not be available. 

Therefore, in such settings and many others, the classical models of two species competition need to be revisited. There is much motivation for non-linear, non-local, or density dependent dispersal in biology and ecology, such as in long range dispersal of birds, seeds, and insects, Levy flights and models with memory eﬀects \cite{Ellefsen-21-2, Ellefsen21-1, Ellefsen23}. Such mechanisms are also applicable in other areas of science and engineering, such as gas dynamics and kinetics, thin film dynamics, plasma physics and nuclear technology \cite{naldi2010mathematical, chapman2015long, stinga2023fractional, song2019spatiotemporal, bonforte2024cauchy, vazquez2006porous, antontsev2015evolution}.
The current manuscript considers a two-species competition model, with both drift and dispersal, where one of the competitors has the capacity to move ``faster". Note, regular diffusion is classically modeled via the Laplacian - however, we model the movement of this ``faster" species via the $p$-Laplacian \cite{antontsev2015evolution}. 
The effect of dispersal modeled via a $p$-Laplacian operator, has been recently considered in several works spanning competitive systems and chemotaxis systems \cite{wang2023global, yang2022global, cong2016degenerate, rani2024global, li2020global}. The problem is ``degenerate" \cite{dibenedetto2012degenerate}, and the analysis herein is more difficult than its regular diffusion counterpart. We present the details of the model next, followed by its analysis, and investigation of dynamics, as well as numerical simulations.

\section{The model formulation}

\subsection{Movement Operator}

Consider a species $v$ dispersing over a spatial domain $\Omega$. Its dynamics are typically governed by a diffusion equation,
$v_{t} = \Delta v + f(v)$, where $\Delta v$ represents movement by diffusion, and the other dynamics such as growth, death, competition, depredation are embedded in $f(v)$. One can define a movement operator $L: H^{2}(\Omega) \mapsto L^{2}(\Omega)$, where $ \mathcal{L}(v) = \Delta v$.
We take the following approach to modeling movement. Define, 
\begin{eqnarray*}
\mathcal{L}^{1}(v)  
&=& \mathcal{L}^{1}((1-k)v + k(v)), \ 0<k<1 \  \nonumber \\
&=& (1-k)L(v) + kL^{*}(v) \nonumber \\
&=& (1-k) \Delta v + k \Delta_{p} v,  \nonumber \\
\end{eqnarray*}
 where $\boxed{L^{*}(v) = \Delta_{p} v =  \nabla \cdot (|\nabla v|^{p-2}  \nabla v) \approx - v^{p-1}}, 1<p\leq 2$. The operator $L^{*}(v)$ will play the role of fast diffusion. However, only a (small) fraction of the population is affected by this motion, which is modeled via the fraction $k \in (0,1)$.

\begin{remark}
The movement operator $\mathcal{L}^{1}$, provides a way to formalize the action in which we have a fraction of the population moving via regular diffusion and the other fraction moving via fast diffusion.
\end{remark}

Based on the above formulation, we can consider a PDE, representing the interaction of two competing species $u(x, t)$ and $v(x, t)$, with spatially dependent growth function $m(x) \in L^{\infty}(\Omega))$,

\begin{equation}
\label{eq:pde_model}
\left\{ \begin{array}{ll}
u_t &= d_1 \nabla \cdot (\nabla u) +  u\Big(m(x)-u-v\Big), \quad x\in \Omega, t>0\\
		v_t &= \nabla \cdot a(x,v,\nabla v)  + v\Big(m(x)-u-v\Big), \quad x\in\Omega,t>0, \\
	 a(x,v,\nabla v) & = \Big( d_2 (1-k)\nabla v  + k |\nabla v|^{p-2}  \nabla v \Big) , \ p \in (1,2], \ 0 \leq k \leq 1 \\
 \nabla u \cdot \eta &=a \cdot \eta=0, \quad x\in \partial \Omega \\
		u(x,0)&=u_0(x), v(x,0)=v_{0} (x), \quad x\in \Omega \\
\end{array}\right.
\end{equation}
where all the parameters $d_i$ $(i=1,2)$ and $k$ are positive and $\Omega \subset \mathbb{R}^n$ is a bounded domain with smooth boundary.

\begin{lemma}
Consider the system \eqref{eq:pde_model}. Then $\nabla a(x,v,\nabla v) \cdot \eta  \iff \nabla v \cdot \eta = 0$.
\end{lemma}
Consider the boundary conditions for $v$,
\[ a \cdot \eta = \Big( d_2 (1-k) \nabla v + k |\nabla v|^{p-2} \cdot \nabla v \Big) \cdot \eta \hspace{.05in} =0 \iff \nabla v \cdot \eta \Big( d_2 (1-k) + k |\nabla v|^{p-2} \Big)=0.\]
By the positivity of the bracket term, we can consider the Neumann boundary conditions for $v$. Consider the system \eqref{eq:pde_model} updated boundary conditions:
\begin{align}\label{pde_model_bc}
	\begin{split}
		 \nabla u \cdot \eta &=\nabla v \cdot \eta=0, \quad x\in \partial \Omega.
	\end{split}
\end{align}

We now state a classical result \cite{dockery1998evolution},
\begin{theorem}[Slower disperser wins]
\label{thm:classic_pde}
Consider \eqref{eq:pde_model} with $p=2$ and  \newline $d_2(1-k) + k < d_1$. Then for any choice of positive initial data $(u_0(x),v_0(x))$, the solution $(u,v)$ converges uniformly to $(0,v^{*})$ as $t\to \infty.$
\end{theorem}
\noindent The proof follows via standard techniques and can be found in \cite{dockery1998evolution, ni2011mathematics}.
In the case $p=2$, solutions to \eqref{eq:pde_model} are classical. However, in the case $1<p<2$, several degeneracies can occur, and one cannot expect a classical solution, even begining from smooth initial conditions. Therefore, we must define weak solutions,

\begin{definition}[Weak solution]\label{def:weak}
A measurable function $v$ is a local weak sub (super) solution of \eqref{eq:pde_model} in $\Omega_{T}$ if 

\begin{equation}
v \in C_{loc}(0,T;L^{2}_{loc}(\Omega) \cap L^{p}_{loc}(0,T;W^{1,p}_{loc}(\Omega) ),
\end{equation}
and for every compact subset $K$ of $\Omega$ and for every subinterval $[t_{1},t_{2}]$ of $(0,T]$,

\begin{equation}
\int_{K} v \phi dx \vert^{t_{2}}_{t_{1}}  + \int^{t_{2}}_{t_{1}} \int_{K} \left(   -v \phi_{t}   +  a(x,v,\nabla v) \cdot \nabla \phi    \right) dx d \tau\geq (\leq)  \int^{t_{2}}_{t_{1}} \int_{K}  b(x,\tau,v) \phi dx d \tau,
\end{equation}
for all test functions $\phi \in W^{1,2}_{loc}(0,T;L^{2}(K) \cap L^{p}_{loc}(0,T;W^{1,p}_{0}(K) ), \ \phi \geq 0$. 
\end{definition}

We next consider the classical competition model, for two competing species $u,v$, subject to a unidirectional drift \cite{lou2014evolution},
\begin{equation}
\label{eq:pde_modeldr}
\left\{ \begin{array}{ll}
u_t = d_1 u_{xx} -q u_x +  u\Big(m-u-v\Big), \quad x\in [0,L], t>0\\
		v_t = d_2 v_{xx} -q v_x  + v\Big(m-u-v\Big), \quad x\in [0,L],t>0, \\
	   d_{1}u_{x}(0,t) - qu(0,t) = d_{2}v_{x}(0,t) - qv(0,t) = 0 , \ u_{x}(L,t)  = v_{x}(L,t) = 0. \\
		u(x,0)=u_0(x), v(x,0)=v_{0} (x), \quad x\in [0,L] \\
\end{array}\right.
\end{equation}
    We recap the classical result \cite{lou2014evolution} pertaining to \eqref{eq:pde_modeldr},

    \begin{theorem}[Faster disperser wins]
    \label{thm:clde}
Consider \eqref{eq:pde_modeldr}, Suppose that $q, m$ are positive constants. If $d_2 > d_1$, then $( 0,v^{*})$, whenever it exists, is globally asymptotically stable.
    \end{theorem}

\subsection{The model with a ``much faster" diffuser}

We now consider the situation where there will be a ``much faster" disperser.


\begin{equation}
\label{eq:pde_modeld2-k}
\left\{ \begin{array}{ll}
u_t &= d_1 \nabla \cdot (\nabla u) - \nabla\cdot (\textbf{q}  u)+  u\Big(m(x)-u-v\Big), \quad x\in \Omega, t>0\\		v_t &= \nabla \cdot a(\nabla v)  - \nabla\cdot (\textbf{q}  v) + v\Big(m(x)-u-v\Big), \quad x\in\Omega,t>0, \\
	 a(\nabla v) & = \Big( d_2 (1-k)\nabla v  + d_2 k  |\nabla v|^{p-2}  \nabla v \Big) , \ p \in (1,2), \ 0 \leq k \leq 1 \\
 (d_{1}\nabla u - \textbf{q} u)  \cdot \eta &= (d_2   |\nabla v|^{p-2}  \nabla v - \textbf{q} v)  \cdot \eta =0, \quad x\in \partial \Omega_{1} \\
 \nabla u \cdot \eta &= |\nabla v|^{p-2} \nabla v \cdot \eta=0, \quad x\in \partial \Omega_{2} \\
		u(x,0)&=u_0(x), v(x,0)=v_{0} (x), \quad x\in \Omega.
\end{array}\right.
\end{equation}
Here we consider the general model in $n=1,2$, for the drift case, with the much faster dispersal, modeled via a $p$-Laplacian operator. The boundary is $\partial \Omega = \partial \Omega_{1} + \partial \Omega_{2}$, where $\partial \Omega_{1}$ is the portion of the boundary that corresponds to the upstream end, and $\partial \Omega_{2}$ corresponds to the downstream end. We impose the classical Danckwerts boundary condition.
When $k=1$, this corresponds to a special case of the above model where there is only ``fast" diffusion in the $v$ species, and we assume constant resource function $m(x) = m$,

\begin{equation}
\label{eq:pde_modeld211}
\left\{ \begin{array}{ll}
u_t &= d_1 \nabla \cdot (\nabla u) - \nabla\cdot (\textbf{q}  u)+  u\Big(m-u-v\Big), \quad x\in \Omega, t>0\\
		v_t &= \nabla \cdot \Big( d_2 |\nabla v|^{p-2}  \nabla v \Big)  - \nabla\cdot (\textbf{q}  v) + v\Big(m-u-v\Big), \quad x\in\Omega,t>0, \\
(d_{1}\nabla u - \textbf{q} u)  \cdot \eta &= (d_2   |\nabla v|^{p-2}  \nabla v - \textbf{q} v)  \cdot \eta =0, \quad x\in \partial \Omega_{1} \\
 \nabla u \cdot \eta &= |\nabla v|^{p-2} \nabla v \cdot \eta=0, \quad x\in \partial \Omega_{2} \\
		u(x,0)&=u_0(x), v(x,0)=v_{0} (x), \quad x\in \Omega .
\end{array}\right.
\end{equation}

\begin{definition}
\label{def:d11}
Let $1<p<2$, let $T \in (0,\infty]$ and $\Omega \subset \mathbb{R}^{n}, n=1,2$ be a bounded domain, with smooth boundary. A pair $(u,v)$ of non-negative functions defined on $\Omega \times (0,T)$ is called a weak solution to \eqref{eq:pde_modeld211} on $[0,T]$ if,

(i) $u \in L^{2}_{loc}([0,T];W^{1,2}(\Omega)), v \in L^{2}_{loc}([0,T];L^{2}(\Omega))$

(ii) $|\nabla v|^{p-2} \nabla v \in L^{1}_{loc}([0,T];L^{1}(\Omega));$

(iii) For every $\phi \in C^{\infty}_{0}(\Omega \times [0,T))$ we have,

\begin{eqnarray}
&& \int^{T}_{0}\int_{\Omega} v \phi_{t} dxdt - \int_{\Omega} v_{0}(x) \phi(x,0)dx \nonumber \\
&& = 
- \int^{T}_{0}\int_{\Omega} |\nabla v|^{p-2} \nabla v \cdot \nabla \phi dx dt + \int^{T}_{0}\int_{\Omega}\textbf{q} v \nabla \phi dx dt + \int^{T}_{0}\int_{\Omega} v(m-u-v) \phi dxdt . \nonumber \\
\end{eqnarray}
\end{definition}

\subsection{The regularized system and preliminary estimates}
There can be several issues with \eqref{eq:pde_modeld211}, and well posedness has to be dealt with. Herein, the analysis is more involved than with Finite Time Extinction (FTE) through a semi-linearity, as in \cite{parshad2021some}. We follow many of the techniques enunciated from \cite{li2020global, rani2024global, cong2016degenerate, dibenedetto2012degenerate}, in the ensuing analysis. To this end, we first posit the regularized system, considering again the special case that is when $k=1$,

\begin{equation}
\label{eq:pde_modeld2}
\left\{ \begin{array}{ll}
u^{\epsilon}_t &= d_1 \nabla \cdot (\nabla u^{\epsilon}) - \nabla\cdot (\textbf{q}  u^{\epsilon}) +  u^{\epsilon}\Big(m-u^{\epsilon}-v^{\epsilon}\Big), \quad x\in \Omega, t>0\\
		v^{\epsilon}_t &= \nabla \cdot a(\epsilon,\nabla v^{\epsilon})  - \nabla\cdot (\textbf{q}  v^{\epsilon}) + v^{\epsilon}\Big(m-u^{\epsilon}-v^{\epsilon}\Big), \quad x\in\Omega,t>0, \\
	 a(\epsilon,\nabla v^{\epsilon}) & = \Big(  d_2   \left(|\nabla v^{\epsilon}|^{2}+{\epsilon}\right)^{\frac{p-2}{2}}  \nabla v^{\epsilon} \Big) , \ p \in (1,2], \ 0 \leq k \leq 1 \\
 (d_{1}\nabla u^{\epsilon} - \textbf{q} u^{\epsilon})  \cdot \eta &= (d_2  \left(|\nabla v^{\epsilon}|^{2}+{\epsilon}\right)^{\frac{p-2}{2}}  \nabla v^{\epsilon} - \textbf{q} v^{\epsilon})  \cdot \eta =0, \quad x\in \partial \Omega_{1} \\
 \nabla u^{\epsilon} \cdot \eta &= \left(|\nabla v^{\epsilon}|^{2}+{\epsilon}\right)^{\frac{p-2}{2}}\nabla v^{\epsilon} \cdot \eta=0, \quad x\in \partial \Omega_{2} \\
		u^{\epsilon}(x,0)&=u^{\epsilon}_0(x), v^{\epsilon}(x,0)=v^{\epsilon}_{0} (x), \quad x\in \Omega . \\
\end{array}\right.
\end{equation}


\noindent We state the following theorem, about the regularized system,

\begin{theorem}
\label{thm:mt1-1}
    Consider \eqref{eq:pde_modeld2}. Let $2>p>\frac{3}{2}, q,m>0$. If the initial data $u_{0}(x) \in W^{1,\infty}(\Omega), v_{0}(x) \in L^{\infty}$, then for any $(d_{1},d_{2})$ and $q \geq 1$, there exists a $C=C(q)$ s.t. for any $ \epsilon > 0$, we have,
\begin{equation}
    ||v^{\epsilon}||_{L^{q}(\Omega)} \leq C,
\end{equation}
  for all $t \in [0, T_{max, \epsilon})$.
\end{theorem}

\begin{proof}
We multiply the equation for $v^{\epsilon}$ in \eqref{eq:pde_modeld2} by $q(v^{\epsilon})^{q-1}$, to obtain,


\begin{eqnarray}
&& \frac{d}{dt}||v^{\epsilon}||^{q}_{q} + d_{2}  q(q-1)\int_{\Omega} \left(
|\nabla v^{\epsilon}|^{2}+{\epsilon}\right)^{\frac{p-2}{2}}|\nabla v^{\epsilon}|^{2} (v^{\epsilon})^{q-2} dx \nonumber \\
&& -\int_{\partial \Omega_{1}} d_{2}\left(
|\nabla v^{\epsilon}|^{2}+{\epsilon}\right)^{\frac{p-2}{2}}( \nabla v^{\epsilon} \cdot \eta ) ((qv^{\epsilon})^{q-1}) dS \nonumber \\
&& -\int_{\partial \Omega_{2}} d_{2}\left(
|\nabla v^{\epsilon}|^{2}+{\epsilon}\right)^{\frac{p-2}{2}}( \nabla v^{\epsilon} \cdot \eta ) ((qv^{\epsilon})^{q-1}) dS \nonumber \\
&&+||v^{\epsilon}||^{1+q}_{1+q} 
 + q\int_{\Omega}  u^{\epsilon}(v^{\epsilon})^{q+1}dx \nonumber \\
 && =mq\int_{\Omega} (v^{\epsilon})^{q}dx + q(q-1)\int_{\Omega} (\textbf{q}\cdot \nabla v^{\epsilon}) (v^{\epsilon})^{q-1}dx \nonumber \\
 && - \int_{\partial \Omega_{1}} (\textbf{q} v^{\epsilon}) \cdot \eta) ((qv^{\epsilon})^{q-1}) dS 
 - \int_{\partial \Omega_{2}} (\textbf{q} v^{\epsilon}) \cdot \eta) ((qv^{\epsilon})^{q-1}) dS 
 \nonumber \\
\end{eqnarray}

However, we can combine the boundary terms so that,

\begin{eqnarray}
 &&   -\int_{\partial \Omega_{1}} \left(
|\nabla v^{\epsilon}|^{2}+{\epsilon}\right)^{\frac{p-2}{2}}( \nabla v^{\epsilon} \cdot \eta ) ((qv^{\epsilon})^{q-1}) dS
+ \int_{\partial \Omega_{1}} (\textbf{q} v^{\epsilon}) \cdot \eta) ((qv^{\epsilon})^{q-1}) dS \nonumber \\
&& = - \int_{\partial \Omega_{1}} ((qv^{\epsilon})^{q-1})  [(d_2  \left(|\nabla v^{\epsilon}|^{2}+{\epsilon}\right)^{\frac{p-2}{2}}  \nabla v^{\epsilon} - \textbf{q} v^{\epsilon})  \cdot \eta] dS \nonumber \\
&& = \int_{\partial \Omega_{1}} ((qv^{\epsilon})^{q-1}) (0)dS = 0, \nonumber \\
\end{eqnarray}
which follows from the Danckwerts boundary condition, so no boundary terms survive over $\partial \Omega_{1}$. Over $\partial \Omega_{2}$ we use the prescribed boundary conditions to yield,

\begin{eqnarray}
&& \frac{d}{dt}||v^{\epsilon}||^{q}_{q} + 
\int_{\partial \Omega_{2}} d_{2}\left(
|\nabla v^{\epsilon}|^{2}+{\epsilon}\right)^{\frac{p-2}{2}}( \nabla v^{\epsilon} \cdot \eta ) ((qv^{\epsilon})^{q-1}) dS \nonumber \\
&&+||v^{\epsilon}||^{1+q}_{1+q} 
 + q\int_{\Omega}  u^{\epsilon}(v^{\epsilon})^{q+1}dx \nonumber \\
 && =mq\int_{\Omega} (v^{\epsilon})^{q}dx + q(q-1)\int_{\Omega} (\textbf{q}\cdot \nabla v^{\epsilon}) (v^{\epsilon})^{q-1}dx \nonumber \\
 &&  
 - \int_{\partial \Omega_{2}} (\textbf{q}  \cdot \eta) (q(v^{\epsilon})^{q}) dS .
 \nonumber \\
\end{eqnarray}
Now note when $1<p<2$, we have,
\begin{equation}
\label{eq:in1}
\int_{\Omega} \left(
|\nabla v^{\epsilon}|^{2}+{\epsilon}\right)^{\frac{p-2}{2}}|\nabla v^{\epsilon}|^{2} (v^{\epsilon})^{q-2} dx
\geq 
\int_{\Omega} \left(
|\nabla v^{\epsilon}|\right)^{p}(v^{\epsilon})^{q-2} dx - \epsilon^{\frac{p}{2}}\int_{\Omega} (v^{\epsilon})^{q-2} dx .
\end{equation}
Also note,
\begin{equation}
\label{eq:in2}
\int_{\Omega} (v^{\epsilon})^{q-2}|\nabla v^{\epsilon}|^{p}dx = \frac{p^{p}}{(p+q-2)^{p}}||\nabla (v^{\epsilon})^{\frac{p+q-2}{p}}||^{p}_{p}.
\end{equation}
Using the above in the $\int_{\Omega} (\textbf{q}\cdot \nabla v^{\epsilon}) (v^{\epsilon})^{q-1}dx$ term, by setting $p=1$, and $q=q^{'}+1$, where $q^{'}$, is a dummy variable, yields, 

\begin{equation}
\int_{\Omega} (v^{\epsilon})^{q-1}|\nabla v^{\epsilon}|^{1}dx = \frac{1}{q}||\nabla (v^{\epsilon})^{q}||^{1}_{1}.
\end{equation}
Using positivity and the earlier estimates via \eqref{eq:in1}-\eqref{eq:in2} entails,


\begin{eqnarray}
&& \frac{d}{dt}||v^{\epsilon}||^{q}_{q}  + d_{2}  q(q-1)[\int_{\Omega} \left(
|\nabla v^{\epsilon}|\right)^{p}(v^{\epsilon})^{q-2} dx - \epsilon^{\frac{p}{2}}\int_{\Omega} (v^{\epsilon})^{q-2} dx]
+q||v^{\epsilon}||^{1+q}_{1+q} \nonumber \\
&& + q\int_{\Omega}  u^{\epsilon}(v^{\epsilon})^{q+1}dx 
+\int_{\partial \Omega_{2}} (\textbf{q}  \cdot \eta) (q(v^{\epsilon})^{q}) dS 
 \nonumber \\
&& 
\leq mq\int_{\Omega} (v^{\epsilon})^{q}dx + q(q-1)\int_{\Omega} (\textbf{q}\cdot \nabla v^{\epsilon}) (v^{\epsilon})^{q-1}dx.\nonumber \\
\end{eqnarray}
Again, using positivity we have,

\begin{eqnarray}
\label{eq:in3}
&& \frac{d}{dt}||v^{\epsilon}||^{q}_{q}  + d_{2}  q(q-1)[\int_{\Omega} \left(
|\nabla v^{\epsilon}|\right)^{p}(v^{\epsilon})^{q-2} dx ]
+q||v^{\epsilon}||^{1+q}_{1+q} \nonumber \\
&& + q\int_{\Omega}  u^{\epsilon}(v^{\epsilon})^{q+1}dx 
 \nonumber \\
&& 
\leq mq\int_{\Omega} (v^{\epsilon})^{q}dx + q(q-1)\int_{\Omega} (\textbf{q}\cdot \nabla v^{\epsilon}) (v^{\epsilon})^{q-1}dx + d_{2}  q(q-1) \epsilon^{\frac{p}{2}}\int_{\Omega} (v^{\epsilon})^{q-2} dx.
\nonumber \\
\end{eqnarray}
Notice,
\begin{equation}
q(q-1)\int_{\Omega} (\textbf{q}\cdot \nabla v^{\epsilon}) (v^{\epsilon})^{q-1}dx = q(q-1)\int_{\Omega} (\textbf{q}\cdot \nabla v^{\epsilon}) (v^{\epsilon})^{\frac{q-2}{p}} (v^{\epsilon})^{q^{*}} ,
\end{equation}
where 
\begin{equation}
    q^{*} = \left(q-1 - \frac{q-2}{p}\right)= \frac{p(q-1) -(q-2)}{p}.
\end{equation}
*
Now, using Holder-Young with $\epsilon_{1}$, and exponents $p, q=\frac{p}{p-1}$, we have,

\begin{equation}
\int_{\Omega} (v^{\epsilon})^{q-1}|\nabla v^{\epsilon}|^{1}dx
\leq \frac{\epsilon_{1}}{p} d_{2} q(q-1)\int_{\Omega} (v^{\epsilon})^{q-2}|\nabla v^{\epsilon}|^{p}dx + \frac{p-1}{p}f(\epsilon_{1})\int_{\Omega} (v^{\epsilon})^{q^{**}}dx.
\end{equation}
Here, $q^{**} = q^{*}\left(\frac{p}{p-1}\right)$, so,

\begin{eqnarray}
   && q^{**} \nonumber \\
   && = \left(q-1 - \frac{q-2}{p}\right)\left( \frac{p}{p-1}\right)  \nonumber \\
   && = \frac{p(q-1) -(q-2)}{p-1} = \frac{(p-1)q + (2-p)}{p-1}  \nonumber \\
   && = q + \left(\frac{2-p}{p-1}\right).  \nonumber \\
\end{eqnarray}
Choosing $\epsilon_{1}$, such that $\epsilon_{1} < p $, and inserting this into \eqref{eq:in3} we have,

\begin{equation}
\label{eq:in44}
 \frac{d}{dt}||v^{\epsilon}||^{q}_{q} 
+q||v^{\epsilon}||^{1+q}_{1+q} 
 \leq mq\int_{\Omega} (v^{\epsilon})^{q}dx +  f(\epsilon_{1})\int_{\Omega} (v^{\epsilon})^{q^{**}}dx + d_{2}  q(q-1) \epsilon^{\frac{p}{2}}\int_{\Omega} (v^{\epsilon})^{q-2} dx.
\end{equation}
Now, if 
\begin{equation}
    q^{**} = q + \left(\frac{2-p}{p-1}\right) < q+1, 
\end{equation} 
which is true if $p > \frac{3}{2}$,  then using the embedding of $L^{q+1}(\Omega) \hookrightarrow L^{q^{**}}(\Omega) \hookrightarrow L^{q}(\Omega) \hookrightarrow L^{q-2}(\Omega)$, and a further use of Holder-Young with, say, $\epsilon_{2}$, choosing $\epsilon_{2} f(\epsilon_{1}) < q$ entails,

\begin{equation}
 \frac{d}{dt}||v^{\epsilon}||^{q}_{q} \leq  C_{1} + C_{2}||v^{\epsilon}||^{q}_{q}
-C_{3}\left(||v^{\epsilon}||^{q}_{q}\right)^{\frac{q+1}{q}}.
\end{equation}
The $L^{q}(\Omega)$ bound on $v^{\epsilon}$, follows, via comparison to the ODE, $y^{'}=C_{1} + C_{2}y - C_{3}y^{1+\frac{1}{q}}$, for $C_{1}, C_{2}, C_{3}, q > 0$.
This proves the theorem.

\end{proof}
\begin{remark}
    Note,  $\boxed{q^{**}=q + \left(\frac{2-p}{p-1}\right)}$, as $p \searrow 1$, $q^{**} >>1$, as large as one wants. Thus, we are not able to uniformly control
    $||v^{\epsilon}||^{q}_{q}$ as the $\int_{\Omega}(v^{\epsilon})^{q^{**}}$ term cannot be absorbed by the 
    $\int_{\Omega}(v^{\epsilon})^{q+1}$ term in \eqref{eq:in44}. This is only possible when $p > \frac{3}{2}$, because in this case, $q^{**} < q+1$.
\end{remark}

\subsection{Alikakos-Moser iteration}

The following lemma can be stated,

\begin{lemma}
\label{lem:ma1}
Consider \eqref{eq:pde_modeld2}. Let $2>p>\frac{3}{2}, d_{1}, d_{2}, q,m>0$, and $\Omega \subset \mathbb{R}^{n}, n=1,2$, s.t. $|\Omega| < \infty$, with smooth boundary. If the initial data $u_{0}(x) \in W^{1,\infty}(\Omega), v_{0}(x) \in L^{\infty}(\Omega)$, then for any $1>>\epsilon > 0$, we have,

\begin{equation}
    ||u^{\epsilon}||_{W^{1,\infty}(\Omega)} \leq C, \ ||v^{\epsilon}||_{L^{\infty}(\Omega)} \leq C,
\end{equation}
for all $t \in [0, T_{max, \epsilon})$  .
\end{lemma}

\begin{proof}

We begin with the estimates for the $v^{\epsilon}$ component.
Here we multiply the $v^{\epsilon}$ equation in the case $k=1$, by $(v^{\epsilon})^{q_{k} - 1}$, and integrate by parts to yield,

\begin{eqnarray}
&& \frac{d}{dt}||v^{\epsilon}||^{q_{k}}_{q_{k}}  
 + d_{2} k q_{k}(q_{k}-1)\int_{\Omega} \left(
|\nabla v^{\epsilon}|^{2}+{\epsilon}\right)^{\frac{p-2}{2}}|\nabla v^{\epsilon}|^{2} (v^{\epsilon})^{q_{k}-2} dx
+||v^{\epsilon}||^{1+q_{k}}_{1+q_{k}} \nonumber \\
&& + q_{k}\int_{\Omega}  u^{\epsilon}(v^{\epsilon})^{q_{k}+1}dx =m q_{k}\int_{\Omega} (v^{\epsilon})^{q_{k}}dx + q_{k}(q_{k}-1)\int_{\Omega} (\textbf{q}\cdot \nabla v^{\epsilon}) (v^{\epsilon})^{q_{k}-1}dx.\nonumber \\
\end{eqnarray}
Using the estimates as in Theorem \ref{thm:mt1-1}, yields,

\begin{eqnarray}
&& \frac{d}{dt}||v^{\epsilon}||^{q_{k}}_{q_{k}}  
 + d_{2} k q_{k}(q_{k}-1)\int_{\Omega} 
\lvert \nabla \left(v^{\epsilon}\right)^{\frac{p + q_{k} - 2}{p}} \rvert ^{p}dx \nonumber \\
&&+||v^{\epsilon}||^{1+q_{k}}_{1+q_{k}}
 + q_{k}\int_{\Omega}  u^{\epsilon}(v^{\epsilon})^{q_{k}+1}dx  \nonumber \\
 && = m q_{k}\int_{\Omega} (v^{\epsilon})^{q_{k}}dx + q_{k}(q_{k}-1)\int_{\Omega} (\textbf{q}\cdot \nabla v^{\epsilon}) (v^{\epsilon})^{q_{k}-1}dx.\nonumber \\
\end{eqnarray}
Now we handle each of the terms in the right-hand side above separately,

\begin{eqnarray}
&& \int_{\Omega} (v^{\epsilon})^{q_{k}}dx  \nonumber \\
&& = ||(v^{\epsilon})^{\frac{p + q_{k} - 2}{p}}||^{\frac{pq_{k}}{p + q_{k} - 2}}_{L^{\frac{pq_{k}}{p + q_{k} - 2}}(\Omega)} \leq ||\nabla ( (v^{\epsilon})^{\frac{p + q_{k} - 2}{p}})||^{\frac{pq_{k}}{p + q_{k} - 2} \cdot b}_{L^{p}(\Omega)}||(v^{\epsilon})^{\frac{p + q_{k} - 2}{p}}||^{\frac{pq_{k}}{p + q_{k} - 2} \cdot (1-b)}_{L^{\frac{p}{p + q_{k} - 2}}(\Omega)} \nonumber \\
&& + C||(v^{\epsilon})^{\frac{p + q_{k} - 2}{p}}||^{\frac{pq_{k}}{p + q_{k} - 2}}_{L^{\frac{pq_{k}}{2(p + q_{k} - 2)}}} \nonumber \\
&& \leq C_{1}\epsilon_{1} ||\nabla ( (v^{\epsilon})^{\frac{p + q_{k} - 2}{p}})||^{p}_{L^{p}(\Omega)} + C_{2}f(\epsilon_{1})\int_{\Omega}|v^{\epsilon}|dx + \left(\int_{\Omega}|(v^{\epsilon})^{\frac{q_{k}}{2}}|dx\right)^{2}. \nonumber \\
\end{eqnarray}
This follows by choosing 
$b=\frac{\frac{p+q_{k}-2}{p} - \frac{p+q_{k}-2}{pq_{k}}}{\frac{p+q_{k}-2}{p} - \frac{1}{p} + \frac{1}{2}}$. Then $b<1$, as long as $p>\frac{4}{3}$. Thus application of Holder-Young, with $\epsilon_{1}$ small enough yields the above.

Also,

\begin{eqnarray}
&&\int_{\Omega} (\textbf{q}\cdot \nabla v^{\epsilon}) (v^{\epsilon})^{q_{k}-1}dx \leq \int_{\Omega} |(\textbf{q}\cdot \nabla v^{\epsilon})| (v^{\epsilon})^{q_{k}-1}dx \nonumber \\
&& \leq 
|\textbf{q}|\int_{\Omega} (v^{\epsilon})^{q_{k}-1}|\nabla v^{\epsilon}|^{1}dx = \frac{1}{q}||\nabla (v^{\epsilon})^{q_{k}}||^{1}_{1}\nonumber \\
&& \leq |\textbf{q}|\frac{\epsilon_{1}}{p} d_{2} q_{k}(q_{k}-1)\int_{\Omega} (v^{\epsilon})^{q_{k}-2}|\nabla v^{\epsilon}|^{p}dx + |\textbf{q}|\frac{p-1}{p}g(\epsilon_{1})\int_{\Omega} (v^{\epsilon})^{q^{**}}dx ,\nonumber \\
\end{eqnarray}
where if $p > \frac{3}{2}$, $q^{**} < 1 + q_{k}$. We next repeat the same estimate as in Theorem \ref{thm:mt1-1} by applying Holder-Young inequality with $\epsilon_{1}$ small enough, to obtain,

\begin{equation}
 \frac{d}{dt}||v^{\epsilon}||^{q_{k}}_{q_{k}} 
+q_{k}||v^{\epsilon}||^{1+q_{k}}_{1+q_{k}} 
 \leq C_{1} +  f(\epsilon_{1})\int_{\Omega} (v^{\epsilon})^{q_{k}+\frac{2-p}{p-1}}dx +  f(\epsilon_{1})C_{2} \left( \int_{\Omega} (v^{\epsilon})^{\frac{q_{k}}{2}}dx\right)^{2} .
\end{equation}
If we choose $p > \frac{3}{2}$, this yields for $q_{k} = 2^{k}$ we have, 
\begin{equation}
 \frac{d}{dt}||v^{\epsilon}||^{q_{k}}_{q_{k}} 
+q_{k}||v^{\epsilon}||^{1+q_{k}}_{1+q_{k}} 
 \leq  ||v^{\epsilon}||^{q_{k}+\delta}_{q_{k}+\delta}  + \left(||v^{\epsilon}||^{\frac{q_{k}}{2}}_{\frac{q_{k}}{2}}\right)^{2} + C_{1}, \ 0<\delta <1.
\end{equation}
Standard embedding of  $L^{q_{k}+1}(\Omega) \hookrightarrow  \hookrightarrow L^{q_{k}+\delta}(\Omega)$, $0<\delta <1$, entails

\begin{equation}
 \frac{d}{dt}||v^{\epsilon}||^{q_{k}}_{q_{k}} 
+C_{1}\left(||v^{\epsilon}||^{q_{k}}_{q_{k}} \right)^{1+\frac{1}{q_{k}}}
 \leq C_{2}\left(||v^{\epsilon}||^{\frac{q_{k}}{2}}_{\frac{q_{k}}{2}}\right)^{2} + C_{3}.
\end{equation}
It follows trivially that,

\begin{equation}
 \frac{d}{dt}||v^{\epsilon}||^{q_{k}}_{q_{k}} 
+C_{1}||v^{\epsilon}||^{q_{k}}_{q_{k}} 
 \leq C_{2}\left(||v^{\epsilon}||^{\frac{q_{k}}{2}}_{\frac{q_{k}}{2}}\right)^{2} + C_{3}.
\end{equation}
Thus, if 

\begin{equation}
M_{K} = \max \left\{||v_{0}||_{\infty}, sup_{t \in (0,T)} \int_{\Omega} |v^{\epsilon}|^{q_{k}} dx \right \} \ \mbox{for fixed} \ T \in (0.T_{max, \epsilon}),
\end{equation}

it is immediate that, 

\begin{equation}
M_{K} \leq C_{1} ||v_{0}||^{q_{k}}_{q_{k}} + C_{2} (M_{K-1})^{2} + C_{3}. 
\end{equation}
Via a recursion and taking the limit as $k \rightarrow \infty$,

\begin{eqnarray}
&&||v^{\epsilon}||_{\infty}  \nonumber \\
&\leq& \limsup_{k \rightarrow \infty} (M_{K})^{\frac{1}{q_{k}}} \nonumber \\
&\leq& 
(C_{1}||v_{0}||^{q_{k}}_{q_{k}} + C_{2} (M_{K-1})^{\frac{K}{K-1}} + C_{3})^{\frac{1}{q_{k}}}  \nonumber \\
&\leq &
||v_{0}||_{\infty} + C_{1}, \nonumber \\
\end{eqnarray}
or recursively, $(M_{K})^{\frac{1}{q_{k}}} \leq \left(C_{1}(M_{0})^{2^{k}}+C_{2}\right)^{\frac{1}{q_{k}}} $.

The $L^{\infty}$ bound on $v^{\epsilon}$ follows. The bound for the $u^{\epsilon}$ component follows via standard estimates, due to the structure of the equation for $u^{\epsilon}$, we can derive a $W^{2,q}(\Omega)$ estimate for the $u^{\epsilon}$ component, and then use the embedding chain, $W^{2,q}(\Omega) \hookrightarrow C^{1}(\Omega) \hookrightarrow W^{1, \infty}(\Omega)$. This proves the lemma.
\end{proof}
\section{Global Existence of Weak Solutions}

In this section, we prove global existence of weak solutions to \eqref{eq:pde_modeld211}. Herein, the strategy is to consider the regularised system \eqref{eq:pde_modeld2}, then derive estimates uniform in the parameter $\epsilon$, and pass to the limit as $\epsilon \rightarrow 0$. We begin by stating the following lemma,

\begin{lemma}
\label{lem:ge11n}
    Consider \eqref{eq:pde_modeld2}. Let $2>p>\frac{3}{2}, d_{1}, d_{2}, q,m>0$. If the initial data $u_{0}(x) \in W^{1,\infty}(\Omega), v_{0}(x) \in L^{\infty}(\Omega)$, then for any $1>>\epsilon > 0$, there exists a constant $C$, depending only on the problem parameters, such that we have,

\begin{equation}
\label{eq:11l}
    ||u^{\epsilon}||_{W^{1,\infty}(\Omega)} \leq C,
\end{equation}

\begin{equation}
\label{eq:12l}
    ||v^{\epsilon}||_{L^{\infty}(\Omega)} \leq C,
\end{equation}

\begin{equation}
\label{eq:13l}
\int^{t}_{0}\int_{\Omega} 
|\nabla v^{\epsilon}(.,s)|^{p}(v^{\epsilon}(.,s))^{q-2} dx ds \leq C
\end{equation}

\begin{equation}
\label{eq:14l}
\int^{t}_{0}\int_{\Omega} 
|\nabla v^{\epsilon}(.,s)|^{p} dx ds \leq C
\end{equation}
\end{lemma}
\begin{proof}
    \eqref{eq:11l}-\eqref{eq:12l} follow from the estimates in Lemma \ref{lem:ma1}. Furthermore, we have the following inequality, after applying the estimates in Lemma \ref{lem:ma1},

    \begin{eqnarray}
   && \frac{d}{dt}||v^{\epsilon}||^{q_{k}}_{q_{k}}  
 + C\int_{\Omega} 
\lvert \nabla \left(v^{\epsilon}\right)^{\frac{p + q_{k} - 2}{p}} \rvert ^{p}dx +||v^{\epsilon}||^{1+q_{k}}_{1+q_{k}} \nonumber \\
&& \leq 
    C_{1} +  f(\epsilon_{1})\int_{\Omega} (v^{\epsilon})^{q_{k}+\frac{2-p}{p-1}}dx +  f(\epsilon_{1})C_{2} \left( \int_{\Omega} (v^{\epsilon})^{\frac{q_{k}}{2}}dx\right)^{2} \nonumber \\
    \end{eqnarray}
    Integration in time $[0,T]$ of the inequality above, gives the integral estimate \eqref{eq:13l} we require. Setting $p=2$ therein gives \eqref{eq:14l}. This proves the lemma.
\end{proof}
We state the following theorem,
\begin{theorem}
\label{thm:vubd}
    Consider \eqref{eq:pde_modeld2}. Let $2>p>\frac{3}{2}, d_{1}, d_{2}, q,m>0$. If the initial data $u_{0}(x) \in W^{1,\infty}(\Omega), v_{0}(x) \in L^{\infty}(\Omega)$, then there exists a function $v \in L^{p}_{loc}((0,\infty), W^{1,p}(\Omega)) \cap L^{\infty}((0,\infty);L^{\infty}(\Omega))$ and $u \in  L^{2}_{loc}((0,\infty), W^{1,2}(\Omega)) \cap L^{\infty}((0,\infty);L^{\infty}(\Omega))$  as well as a $\Gamma \in L^{\frac{p}{p-1}}_{loc}((0,\infty);L^{\frac{p}{p-1}}(\Omega)) \cap L^{\infty}((0,\infty);L^{\infty}(\Omega))$, and a sequence of approximants $\epsilon = \epsilon_{j} \searrow 0$, such that,

    \begin{equation}
u^{\epsilon} \overset{*}{\rightharpoonup} u^{*} \ \mbox{in} \ L^{\infty}((0,\infty);L^{\infty}(\Omega))
\end{equation}

\begin{equation}
\nabla u^{\epsilon} \overset{*}{\rightharpoonup} \nabla u^{*}  \ \mbox{in} \ L^{\infty}((0,\infty);L^{\infty}(\Omega))
\end{equation}

\begin{equation}
 u^{\epsilon} \rightharpoonup  u^{*}  \ \mbox{in} \ L^{2}_{loc}((0,\infty);W^{1,2}(\Omega))
\end{equation}

 \begin{equation}
v^{\epsilon} \overset{*}{\rightharpoonup} v^{*}  \ \mbox{in} \ L^{\infty}((0,\infty);L^{\infty}(\Omega))
\end{equation}

 \begin{equation}
 \label{eq:w1pv}
\nabla v^{\epsilon} \rightharpoonup \nabla v^{*}  \ \mbox{in} \ L^{p}_{loc}((0,\infty);L^{p}(\Omega))
\end{equation}

 \begin{equation}
|\nabla v^{\epsilon}|^{p-2}\nabla v^{\epsilon} \rightharpoonup  \Gamma \ \mbox{in} \ L^{\frac{p}{p-1}}_{loc}((0,\infty);L^{\frac{p}{p-1}}(\Omega))
\end{equation}

\end{theorem}

\begin{proof}
    The proof follows via the uniform bounds in Lemma \ref{lem:ge11n}.
\end{proof}

Next, we state a result about the time derivative of the solution,

\begin{theorem}
\label{thm:w1s}
    Consider \eqref{eq:pde_modeld2}. Let $2>p>\frac{3}{2}, d_{1}, d_{2}, q,m>0$. If the initial data $u_{0}(x) \in W^{1,\infty}(\Omega), v_{0}(x) \in L^{\infty}(\Omega)$, then there exists a constant $C$ such that for any time $T > 0$, we have,

     \begin{equation} 
\left \lVert \frac{\partial v^{\epsilon}}{\partial t}\right \rVert_{L^{1}((0,T); (W^{1,p}(\Omega))^{*})} \leq C
\end{equation}

\end{theorem}

\begin{proof}
We consider a test function $\zeta \in C^{\infty}_{0}(\Omega)$ such that $||\zeta||_{W^{1,p}(\Omega)} \leq 1$. Note that,

 \begin{equation} 
\left \lVert \frac{\partial v^{\epsilon}}{\partial t}\right \rVert_{L^{1}((0,T); (W^{1,p}(\Omega))^{*})} = \int^{T}_{0}\left( \sup_{\zeta \in C^{\infty}_{0}(\Omega),||\zeta||_{W^{1,p}(\Omega)} \leq 1 } \int_{\Omega} \frac{\partial v^{\epsilon}}{\partial t} \zeta dx\right) dt.
\end{equation}
We now consider,

\begin{eqnarray}
    && \int_{\Omega} \frac{\partial v^{\epsilon}}{\partial t} \zeta dx \nonumber \\
    && = \ \int_{\Omega} [\nabla \cdot a(x,v^{\epsilon},\nabla v^{\epsilon})  - \nabla\cdot (\textbf{q}  v^{\epsilon}) + v^{\epsilon}\Big(m-u^{\epsilon}-v^{\epsilon}\Big)] \zeta dx \nonumber \\
    && = \int_{\Omega} [\nabla \cdot \Big(d_2 k  \left(|\nabla v^{\epsilon}|^{2}+{\epsilon}\right)^{\frac{p-2}{2}}  \nabla v^{\epsilon} \Big)  - \nabla\cdot (\textbf{q}  v^{\epsilon}) + v^{\epsilon}\Big(m-u^{\epsilon}-v^{\epsilon}\Big) ]\zeta dx \nonumber \\
    && \leq \int_{\Omega} |\nabla \zeta| |v^{\epsilon}| dx + |m|\int_{\Omega} | \zeta| |v^{\epsilon}| dx + \int_{\Omega} | \zeta| |v^{\epsilon}|^{2} dx + \int_{\Omega} | \zeta| |v^{\epsilon}||u^{\epsilon}| dx  \nonumber \\
    && + \int_{\Omega}\Big(d_2 k  \left(|\nabla v^{\epsilon}|^{2}+{\epsilon}\right)^{\frac{p-2}{2}}  \nabla v^{\epsilon} \Big) \nabla \zeta dx \nonumber \\
    &&  \leq \int_{\Omega} |\nabla \zeta| |v^{\epsilon}| dx + |m|\int_{\Omega} | \zeta| |v^{\epsilon}| dx + \int_{\Omega} | \zeta| |v^{\epsilon}|^{2} dx + \int_{\Omega} | \zeta| |v^{\epsilon}||u^{\epsilon}| dx  \nonumber \\
    &&+ \int_{\Omega}\Big(d_2 k  \left(|\nabla v^{\epsilon}|^{2}+{\epsilon}\right)^{\frac{p-1}{2}}  \Big) \nabla \zeta dx \nonumber \\
    && \leq C_{1} + C_{2} + C_{3} + C_{4} + C_{5}.
\end{eqnarray}
This follows via the uniform estimates in Lemma \ref{lem:ge11n}, Lemma \ref{lem:pc1}, as well as the assumptions on the test function $\zeta$.
\end{proof}

\begin{lemma}
\label{lem:vlp}
Consider \eqref{eq:pde_modeld2}, with $ 1 < p \leq 2$, there exists a subsequence ${\epsilon_{j}}$, such that,
 \begin{equation}
v^{\epsilon_{j}} \rightarrow v^{*} \ \mbox{in} \  L^{p}([0,T];L^{p}(\Omega))
\end{equation}

\end{lemma}

\begin{proof}
    We know that via estimates \eqref{eq:14l} in Lemma \ref{lem:ge11n} that we have,
    \begin{equation}
v^{\epsilon_{j}} \in L^{p}((0,T), W^{1,p}(\Omega)).
\end{equation}
\end{proof}
Furthermore, via Theorem \ref{thm:w1s} we have,

 \begin{equation}
 \frac{\partial v^{\epsilon_{j}}}{\partial t} \in L^{1}((0,T); (W^{1,p}(\Omega))^{*}).
\end{equation}
We now use the classical Aubin-Lion compactness Lemma \cite{robinson2003infinite}, 

 \begin{equation}
W^{1,p}(\Omega) \hookrightarrow \hookrightarrow L^{p}(\Omega) \hookrightarrow \left(W^{1,p}(\Omega)\right)^{*},
 \end{equation}
 to obtain $v^{\epsilon_{j}} \rightarrow v^{*}$ in 
 $L^{p}((0,T), L^{p}(\Omega))$.

We derive, 
\begin{equation}
\label{eq:56e}
\lim_{j \rightarrow \infty}  \int^{T}_{0}\int_{\Omega}|v^{\epsilon_{j}}-v^{*}|dxdt \rightarrow 0 .
 \end{equation}
Note, as $\lim_{j \rightarrow \infty}\epsilon_{j} \rightarrow 0$. This follows via 
\begin{equation}
L^{p}((0,T), L^{p}(\Omega)) \hookrightarrow L^{1}((0,T), L^{1}(\Omega)).
 \end{equation}

\subsection{Convergence of Non-linear term} 

\begin{lemma}
\label{lem:cnt}
Consider \eqref{eq:pde_modeld2} with $1<p<2$. There exists a subsequence ${\epsilon_{j}}$, such that for any $T>0$,

\begin{equation}
|\nabla v^{\epsilon_{j}}|^{p-2}  \nabla v^{\epsilon_{j}} \rightarrow |\nabla v|^{p-2}  \nabla v \ \mbox{in} \ L^{\frac{p}{p-1}}(\Omega \times (0,T)).
 \end{equation}
\end{lemma}
We will first prove convergences in the regularized system \eqref{eq:pde_modeld2}. We first consider,

\begin{eqnarray}
 && \lim_{j \rightarrow \infty} \int^{T}_{0} \int_{\Omega} \frac{\partial v^{\epsilon_{j}}}{\partial t} \phi dx dt \nonumber \\
 && = - \lim_{j \rightarrow \infty}\int^{T}_{0} \int_{\Omega} \frac{\partial \phi}{\partial t} v^{\epsilon_{j}} dx dt + \int_{\Omega}v^{\epsilon_{j}}_{0}(x)\phi_{0}(x)dx \nonumber \\
 && =\int^{T}_{0} \int_{\Omega} \frac{\partial \phi}{\partial t} v dx dt + \int_{\Omega}v_{0}(x)\phi_{0}(x)dx . \  \nonumber \\
\end{eqnarray}
This follows as $\nabla v^{\epsilon_{j}} \rightharpoonup \nabla v^{*}  \ \mbox{in} \ L^{p}_{loc}((0,\infty);L^{p}(\Omega))$, and $L^{p}((0,T);W^{1,p}(\Omega)) \hookrightarrow L^{1}((0,T);L^{p}(\Omega))$.

The semilinear term convergences follow similarly, we demonstrate with the linear term,

\begin{eqnarray}
 && \lim_{j \rightarrow \infty} \int^{T}_{0} \int_{\Omega}  m v^{\epsilon_{j}} \phi dx dt \nonumber \\
 && =\int^{T}_{0} \int_{\Omega} m v \phi dx dt + \int_{\Omega}m v_{0}(x)\phi_{0}(x)dx . \  \nonumber \\
\end{eqnarray}
This follows again from \eqref{eq:56e} and $L^{p}((0,T);W^{1,p}(\Omega)) \hookrightarrow L^{1}((0,T);L^{p}(\Omega))$.

For convergence of the quadratic term, $\lim_{j \rightarrow \infty} \int^{T}_{0} \int_{\Omega}  m (v^{\epsilon_{j}})^{2} \phi dx dt$, we appeal to Lemma \ref{lem:vlp} and $W^{1,p}(\Omega) \hookrightarrow \hookrightarrow L^{2}(\Omega) $, for $n=1,2$. The convergence of the drift term is handled similarly, via an integration by parts in space. Thus we have,

\begin{eqnarray}
    && - \int^{T}_{0}\int_{\Omega} v \frac{\phi }{\partial t}  dx  - \int_{\Omega}v_{0}(x)\phi_{0}(x)dx \nonumber \\
    && = - \int^{T}_{0}\int_{\Omega} \Gamma \cdot \nabla \phi dx dt  - \int^{T}_{0}\int_{\Omega}\textbf{q}  v \cdot \nabla \phi dx dt + \int^{T}_{0}\int_{\Omega}[v\Big(m-u-v\Big)]\phi dx dt. \nonumber \\
   \end{eqnarray}

Next, we show,

\begin{equation}
    \lim_{j \rightarrow \infty} \int^{T}_{0}\int_{\Omega}  \Gamma^{\epsilon_{j}} \cdot \nabla \phi dx dt = \int^{T}_{0}\int_{\Omega}  \Gamma \cdot \nabla \phi dx dt = 
    \int^{T}_{0}\int_{\Omega} |\nabla v|^{p-2} \nabla v \cdot \nabla \phi dx dt .
\end{equation}
Thus, we will show that 
\begin{equation}
\label{eq:11el}
 \int^{T}_{0}\int_{\Omega} (|\nabla v|^{p-2} \nabla v - \Gamma ) \cdot \nabla \phi dx dt = 0,
\end{equation}
both in $(\geq)$ and $(\leq)$ settings.

\begin{lemma}
Consider a function $\phi \in L^{p}((0,T);W^{1,p}(\Omega))$, then for solutions to \eqref{eq:pde_modeld2} we have,

\begin{equation}
\label{eq:iiel}
     \int^{T}_{0}\int_{\Omega} ( \Gamma - |\nabla \phi|^{p-2} \nabla \phi  ) \cdot (\nabla v - \nabla \phi )dx dt \geq  0.
\end{equation}
    
\end{lemma}

\begin{proof}
We consider,
\begin{eqnarray}
\label{eq:i1i}
&& \int^{T}_{0}\int_{\Omega} ( |\nabla v^{\epsilon}|^{p-2} \nabla v^{\epsilon}  - |\nabla \phi|^{p-2} \nabla \phi  ) \cdot (\nabla v - \nabla \phi )dx dt \nonumber \\
    && = \int^{T}_{0}\int_{\Omega} ( |\nabla v^{\epsilon}|^{p-2} \nabla v^{\epsilon} \cdot (\nabla v - \nabla v^{\epsilon})dxdt
    +\int^{T}_{0}\int_{\Omega}|\nabla \phi|^{p-2} \nabla \phi  ) \cdot (\nabla v^{\epsilon} - \nabla v)dxdt \nonumber \\
    && + \int^{T}_{0}\int_{\Omega} ( |\nabla v^{\epsilon}|^{p-2} \nabla v^{\epsilon}  - |\nabla \phi|^{p-2} \nabla \phi  ) (\nabla v^{\epsilon} - \nabla \phi)dxdt\nonumber \\
    && = I_{1} + I_{2} + I_{3} .\nonumber \\
\end{eqnarray}
    We see $I_{2} \rightarrow 0$ as ($\epsilon_{j}$ relabeled $\epsilon$) $\epsilon \rightarrow 0$, via Theorem \ref{thm:vubd} and, in particular, \eqref{eq:w1pv}. Furthermore $I_{3} \geq 0$ via Lemma \ref{lem:pc1}. Next, we tackle the $I_{1}$ term. We multiply \eqref{eq:pde_modeld2} by $v-v^{\epsilon}$ and integrate by parts to yield,

\begin{eqnarray}
&& \int^{T}_{0}\int_{\Omega} \left( \frac{\partial v^{\epsilon}}{\partial t}  \right) ( v -  v^{\epsilon} )dx dt \nonumber \\
&& = \int^{T}_{0}\int_{\Omega} ( |\nabla v^{\epsilon}|^{2}  + \epsilon )^{\frac{p-2}{2}} \nabla v^{\epsilon}  ( \nabla (v -   v^{\epsilon}) )dx dt \nonumber \\
&& - \int^{T}_{0}\int_{\Omega} \nabla\cdot (\textbf{q}  v^{\epsilon}) ( v -  v^{\epsilon} )dx dt
+\int^{T}_{0}\int_{\Omega} v^{\epsilon}\Big(m-u^{\epsilon}-v^{\epsilon}\Big) ( v -  v^{\epsilon} )dx dt. \nonumber \\
\end{eqnarray}
     Thus, we have that as ($\epsilon_{j}$, $j \rightarrow \infty$, relabeled as $\epsilon$) $\epsilon \searrow 0$,

\begin{eqnarray}
&& \lim_{\epsilon \rightarrow 0} \int^{T}_{0}\int_{\Omega} \left( \frac{\partial v^{\epsilon}}{\partial t}  \right) ( v -  v^{\epsilon} )dx dt \nonumber \\
&& = \lim_{\epsilon \rightarrow 0} \int^{T}_{0}\int_{\Omega} ( |\nabla v^{\epsilon}|^{2}  + \epsilon )^{\frac{p-2}{2}} \nabla v^{\epsilon}  ( \nabla (v -  v^{\epsilon}))dx dt \nonumber \\
&& - \lim_{\epsilon \rightarrow 0} \int^{T}_{0}\int_{\Omega} \nabla\cdot (\textbf{q}  v^{\epsilon}) ( v -  v^{\epsilon} )dx dt
+ \lim_{\epsilon \rightarrow 0} \int^{T}_{0}\int_{\Omega} v^{\epsilon}\Big(m-u^{\epsilon}-v^{\epsilon}\Big) ( v -  v^{\epsilon} )dx dt. \nonumber \\
\end{eqnarray}
 Using Theorem \ref{thm:w1s} and estimates in Lemma \ref{lem:vlp}, we have the left hand side approaches 0. Whereas via Lemma \ref{lem:ge11n} and  and Lemma \ref{lem:vlp}, we have the second term approaches 0. Lastly, Lemma \ref{lem:ge11n} and  and Lemma \ref{lem:vlp} enable the third term on the right hand side approach zero, and,

 \begin{equation}
\lim_{\epsilon \rightarrow 0} \int^{T}_{0}\int_{\Omega} ( |\nabla v^{\epsilon}|^{2}  + \epsilon )^{\frac{p-2}{2}} \nabla v^{\epsilon} \cdot  ( \nabla ( v -   v^{\epsilon} ))dx dt =  \int^{T}_{0}\int_{\Omega} ( |\nabla v^{\epsilon}|^{p-2} \nabla v^{\epsilon} \cdot (\nabla v - \nabla v^{\epsilon}).
    \end{equation}
 Thus,

 \begin{equation}
0 =  \int^{T}_{0}\int_{\Omega} ( |\nabla v^{\epsilon}|^{p-2} \nabla v^{\epsilon} \cdot (\nabla v - \nabla v^{\epsilon}) + 0 + 0,
    \end{equation}
    which we now inject back into \eqref{eq:iiel} to yield $I_{1} \rightarrow 0$. Now we can say that \eqref{eq:iiel} is true, thus choosing $\psi \in L^{p}(0,T;W^{1,p}(\Omega))$ we have that $\lambda \psi = u - \phi$, with $\lambda \in \mathbb{R}$. Inserting this form into \eqref{eq:iiel} we obtain,

    \begin{equation}
\int^{T}_{0}\int_{\Omega}(\Gamma - |\nabla v - \lambda \nabla \psi|^{p-2} \nabla (v - \lambda  \psi)) \cdot \nabla \psi   dxdt \geq 0.
    \end{equation}
    Letting $\lambda \rightarrow 0^{+}$ and $\lambda \rightarrow 0^{-}$, we obtain the desired equality. This proves the theorem.
\end{proof}

Now we state our main result,

\begin{theorem}
\label{thm:mt1}
Consider the system \eqref{eq:pde_modeld211}, with $2>p>\frac{3}{2}, d_{1}, d_{2}, q,m>0$. If the initial data $u_{0}(x) \in W^{1,\infty}(\Omega), v_{0}(x) \in L^{\infty}(\Omega)$, then there exists a global weak solution $(u,v)$ to system \eqref{eq:pde_modeld211}, in the sense of Definition \ref{def:d11}. Furthermore, there exists a constant $C$, independent of time and initial data, such that

\begin{equation}
   ||u||_{W^{1,\infty}(\Omega)} \leq C, \ ||v||_{L^{\infty}(\Omega)} \leq C, \
  \end{equation}
  for all $t > 0$.
\end{theorem}

\begin{proof}
The proof follows via Lemma \ref{lem:ge11n}, Theorem \ref{thm:vubd}, Theorem \ref{thm:w1s}, Lemma \ref{lem:vlp} and Lemma \ref{lem:cnt}.
\end{proof}

\section{The faster diffuser ``loses"}

\subsection{The case: $\frac{3}{2} < p <2$}

\begin{theorem}
\label{thm:FFTEdd}
Consider the spatially explicit competition model with drift \eqref{eq:pde_modeld211}.  Then, for all positive initial data $(u_0(x),v_0(x))$ if $d_2> d_1$ and $p=2$, the solution $(u,v) \to (0,v^{*})$. However there exists some positive initial data $(u_0(x),v_0(x))$, and some $p \in(\frac{3}{2},2]$,  $ (u,v) \to (u^{*},0)$ in $L^{2}(\Omega)$.
\end{theorem}

\begin{proof}
Consider system $(\ref{eq:pde_model})$, and the equation for for $v$. Let's test the equation for $v$ in  $(\ref{eq:pde_model})$, against $v$ itself,
\[ \int_{\Omega}v v_t =\int_{\Omega}\left( d_2  v v_{xx}  + m v^2 -v^2u-v^3 + k v \dfrac{\partial}{\partial x} \Big(|v_x|^{p-2} \cdot v_x \Big) \right).\]
Integrating over the full domain $\Omega$ and using the boundary conditions $(\ref{pde_model_bc})$, we get,
\[ \dfrac{1}{2} \dfrac{d}{dt} ||v||_2^2 +  ||v||_3^3 + \int_\Omega uv^2 + k ||v_x||_p^p \le ||m||_{\infty}||v||_2^2 + ||v||_{2 + \frac{2-p}{p-1}}^{2 + \frac{2-p}{p-1}} .\]
Using the positivity of $v$ and the fact $m\in \mathcal{L}^{\infty} (\Omega)$, we have,
\[ \dfrac{1}{2} \dfrac{d}{dt} ||v||_2^2 +  ||v||_3^3 + C_{1}||v_x||_p^p \le M ||v||_2^2 + ||v||_{2 + \frac{2-p}{p-1}}^{2 + \frac{2-p}{p-1}},\]
where $M=||m||_{\infty}.$
We first use the weighted $L^{p}$ inequality via Lemma \ref{lem:wlp} where,
\[  L^{3}(\Omega) \dhookrightarrow L^{2 + \frac{2-p}{p-1}}(\Omega) \dhookrightarrow L^{2} (\Omega), .\]
followed by Young's inequality
to obtain,
\begin{equation}
    ||v||_{2 + \frac{2-p}{p-1}} \leq C_{1}||v||^{2}_{3} + C_{2}||v||^{2}_{2} \leq C_{3} + C_{4}||v||^{3}_{3} + C_{2}||v||^{2}_{2},
\end{equation}
such that,

\begin{equation}
    \dfrac{1}{2} \dfrac{d}{dt} ||v||_2^2 +  C_{5}||v||_3^3 + C ||v_x||_p^p \le C_{4}||v||_2^2 + C_{3} .
\end{equation}

Now for $\theta=\frac{1}{2}$, we apply Lemma \ref{lem:gns} to avail,
\begin{equation}\label{GNS_pde}
    ||v||_{2} \le C ||v_x||^{\theta}_{p} ||v||^{1-\theta}_{3}.
\end{equation}
Let's raise the both sides of $(\ref{GNS_pde})$ by $l$, where $l \in (0,2)$
\[ \Big( \int_\Omega v^2 \Big)^{\frac{l}{2}} \le C \Big( \int_\Omega |v_x|^{p} \Big)^{\frac{l \theta}{p}} \Big( \int_\Omega |v|^3 \Big)^{\frac{l(1-\theta)}{3}}. \]
Recall Young's inequality \cite{evans},
\[ ab \le \dfrac{a^r}{r} + \dfrac{b^s}{s}, \]
such that $\frac{1}{r} + \frac{1}{s}=1.$ 
Let's use the Young's inequality for $r = \frac{p}{l \theta}$ and $s=\frac{p}{l(1-\theta)}$, where $\theta = \frac{1}{2}$.  Moreover, $p \in (1,2]$, so we can find a $l \in (0,2)$ such that,
\begin{equation}\label{ODE}
    Y_t \le C_{3} + MY-\widetilde{C} Y^{\alpha},
\end{equation}
where $Y=||v||_2$. As $p\in (1,2]$, we can fix $\alpha =\frac{p}{2}\in (0,1).$ Via similar methods as in \cite{parshad2021some}, we can prove that $\exists T^*<\infty$ such that $Y \to 0$ as $T \to T^*$.


\end{proof}

\begin{remark}
One direct implication of the result of Theorem \ref{thm:classic_pde} along with Theorem \ref{thm:FFTEdd} is that we can say:
The faster diffusing population wins, in the case of drift - but a ``much faster diffusing" population could lose for certain initial conditions. Thus in the case of a system with drift and a ``much faster" disperser - the slower mover can actually have an advantage.
\end{remark}

\begin{corollary}\label{cor:FFTEddc}
Consider the spatially explicit competition model \eqref{eq:pde_model}-\eqref{pde_model_bc} in a bounded domain with smooth boundary $\Omega \subset \mathbb{R}$. If $v(0,t) \leq v(L,t)$, there exists some positive initial data $(u_0(x),v_0(x))$ such that for $d_2(1-k) + k < d_1$ and $p=2$, the solution $(u,v) \to (0,v^{*}),$ but for some $p \in(1,2]$ and $k\in (0,1)$, $ (u,v) \to (u^{*},0)$ in $L^{2}(\Omega)$, starting from the same initial data. 
\end{corollary}

\begin{remark}
One direct implication of the result of Corollary  \ref{cor:FFTEddc} along with Theorem \ref{thm:FFTEdd} is that we can say:
The faster dispersing population wins in the case of drift - but a faster diffusing population with a few ``much faster" dispersers could lose for certain initial conditions. 
\end{remark}

\subsection{The case: $\frac{4}{3} < p <\frac{3}{2} $}

We will first provide an auxilliary lemma,

\begin{lemma}
\label{lem:l1e}
Consider \eqref{eq:pde_modeld2}. Let $\frac{3}{2} > p > \frac{4}{3}, q,m>0$, and $\Omega \subset \mathbb{R}^{n}, n=1,2$, s.t. $|\Omega| < \infty$, with smooth boundary. If the initial data $u_{0}(x) \in W^{1,\infty}(\Omega), v_{0}(x) \in L^{\infty}$ is sufficiently small, then for any $(d_{1},d_{2})$ and $q \geq 1$,there exists a $C=C(q)$ s.t. for any $1>>\epsilon > 0$, we have,

\begin{equation}
    ||v^{\epsilon}||_{L^{q}(\Omega)} \leq C,
\end{equation}
  for all $t \in [0, T_{max, \epsilon})$  .
\end{lemma}

\begin{proof}
From Lemma \ref{lem:ma1}, we have,

\begin{equation}
 \frac{d}{dt}||v^{\epsilon}||^{q}_{q} 
+q||v^{\epsilon}||^{1+q}_{1+q} 
 \leq mq\int_{\Omega} (v^{\epsilon})^{q}dx +  f(\epsilon_{1})\int_{\Omega} (v^{\epsilon})^{q + \left(\frac{2-p}{p-1}\right)}dx.
\end{equation}
Setting $p=\frac{4}{3}$, we have ,

\begin{equation}
 \frac{d}{dt}||v^{\epsilon}||^{q}_{q} 
+q||v^{\epsilon}||^{1+q}_{1+q} 
 \leq mq\int_{\Omega} (v^{\epsilon})^{q}dx +  f(\epsilon_{1})\int_{\Omega} (v^{\epsilon})^{q+2}dx.
\end{equation}
Setting $Y(t)=||v^{\epsilon}||^{q}_{q}$, standard embeddings of $L^{q+1}(\Omega) \hookrightarrow L^{q}(\Omega)$ yields,

\begin{equation}
 \frac{d}{dt}Y \leq mq Y + F(Y) - CqY^{\left(\frac{1+q}{q}\right)}.
\end{equation}
Here, $F(Y)$ is a super linear source term. Standard comparison with the ODE,

\begin{equation}
 \frac{d}{dt}Y = C_{1} Y + F(Y) - C_{2} Y^{\left(\frac{1+q}{q}\right)}
\end{equation}
for some constants $C_{1},C_{2}>0$, yields global existence, for sufficiently small initial data. 
\end{proof}

\begin{remark}
The result of Lemma \ref{lem:l1e} says we have small data global existence of weak solutions to \eqref{eq:pde_modeld211}, even when $\frac{3}{2} > p $. Herein, we apply the same techniques as earlier and pass to the limit as $\epsilon \rightarrow 0$, following previous estimates via Lemma \ref{lem:ma1}.
\end{remark}
This lemma will be used next. We state the following theorem,

\begin{theorem}
\label{thm:FFTEdd1}
Consider the spatially explicit competition model with drift \eqref{eq:pde_modeld211}.  There exists some positive initial data $(u_0(x),v_0(x))$ such that for $d_2> d_1$ and $p=2$, the solution $(u,v) \to (0,v^{*}),$ but for some $p \in(\frac{4}{3},\frac{3}{2})$  $ (u,v) \to (u^{*},0)$ in $L^{2}(\Omega)$, starting from the same initial data. 
\end{theorem}

\begin{proof}
    Consider the system $(\ref{eq:pde_model})$ for $v$, 
Let's test the PDE $(\ref{eq:pde_model})$ for $v$ against $v$
\[ \int_{\Omega}v v_t =\int_{\Omega}\left( d_2  v v_{xx}  + m v^2 -v^2u-v^3 + k v \dfrac{\partial}{\partial x} \Big(|v_x|^{p-2} \cdot v_x \Big) \right).\]
Integrating over the full domain $\Omega$ and using the boundary conditions $(\ref{pde_model_bc})$, we get,
\[ \dfrac{1}{2} \dfrac{d}{dt} ||v||_2^2 +  ||v||_3^3 + \int_\Omega uv^2 + k ||v_x||_p^p \le ||m||_{\infty}||v||_2^2 + ||v||_{2 + \frac{2-p}{p-1}}^{2 + \frac{2-p}{p-1}} .\]
Using the positivity of $v$ and the fact $m\in L^{\infty} (\Omega)$, we have,
\[ \dfrac{1}{2} \dfrac{d}{dt} ||v||_2^2 +  ||v||_3^3 + C_{1}||v_x||_p^p \le M ||v||_2^2 + ||v||_{2 + \frac{2-p}{p-1}}^{2 + \frac{2-p}{p-1}},\]
where $M=||m||_{\infty}.$

Now we choose exponents such that, 

\[ n=2, \quad  k=0, \quad p'=2 + \frac{2-p}{p-1}, \quad m=1, \quad q'=p, \quad q=2,\]
and
\[ -\dfrac{2}{2+\frac{2-p}{p-1}} = \theta \Big( 1 - \dfrac{2}{p}  \Big) - (1-\theta)).\]
On further rearrangement, we have
\begin{equation}\label{GNS_est}
 0<   \theta = \dfrac{(2-p)}{2p-2}.
\end{equation}
We always have $\theta < 1$, when $p>\frac{4}{3}$ .
Applying Lemmma \ref{lem:gns}, 
\begin{equation}\label{GNS_pde2}
    ||v||_{2+\frac{2-p}{p-1}} \le C ||v_x||^{\theta}_{p} ||v||^{1-\theta}_{2}.
\end{equation}
Now let's raise the both sides of $(\ref{GNS_pde2})$ by $l$, where $l \in (0,2)$
\[ \Big( \int_\Omega v^{2 + \frac{2-p}{p-1}} \Big)^{\frac{l}{2 + \frac{2-p}{p-1}}} \le C \Big( \int_\Omega |v_x|^{p} \Big)^{\frac{l \theta}{\widehat{p}}} \Big( \int_\Omega |v|^2 \Big)^{\frac{l(1-\theta)}{2}}. \]
Again, via Young's inequality \cite{evans},
\[ ab \le \dfrac{a^r}{r} + \dfrac{b^s}{s}, \] $\frac{1}{r} + \frac{1}{s}=1$ , with $r = \frac{p}{l \theta}$ and $s=\frac{p}{l(1-\theta)}$, where $\theta = \frac{1}{2}$ and Moreover, $p \in (1,2]$, so we can find a $l \in (0,2)$ such that
we have a differential inequality of the form,
\begin{equation}\label{ODE2}
    Y_t \le C_{3} + MY-\widetilde{C} Y^{\alpha},
\end{equation}
where $Y=||v||_2$. As $p\in (1,2]$, we can fix $\alpha =\frac{p}{2}\in (0,1).$ Via similar methods as in \cite{parshad2021some}, we can prove that $\exists T^*<\infty$ such that $Y \to 0$ as $T \to T^*$.

\end{proof}

\section{Numerical Results}

We perform numerical simulations to explore the dynamics of system \eqref{eq:pde_modeld211} . 
In order to solve system \eqref{eq:pde_modeld211} numerically, we employ a simple finite difference method in space for $N=300$ grid points and a Runge-Kutta solver in time. To avoid places where $(\nabla u)^{p-2}$ is undefined, we perform simulations for the regularized system \eqref{eq:pde_modeld2} with $\epsilon = 10^{-4}$. In each simulation, we model two species $u$ and $v$ where $v$ has a larger diffusion coefficient, $d_2>d_1$. Therefore, we refer to $v$ as the fast diffuser.

\begin{figure}[h!]
    \centering
    \begin{subfigure}[t]{0.3\textwidth}
        \centering
        \includegraphics[height=1.2in]{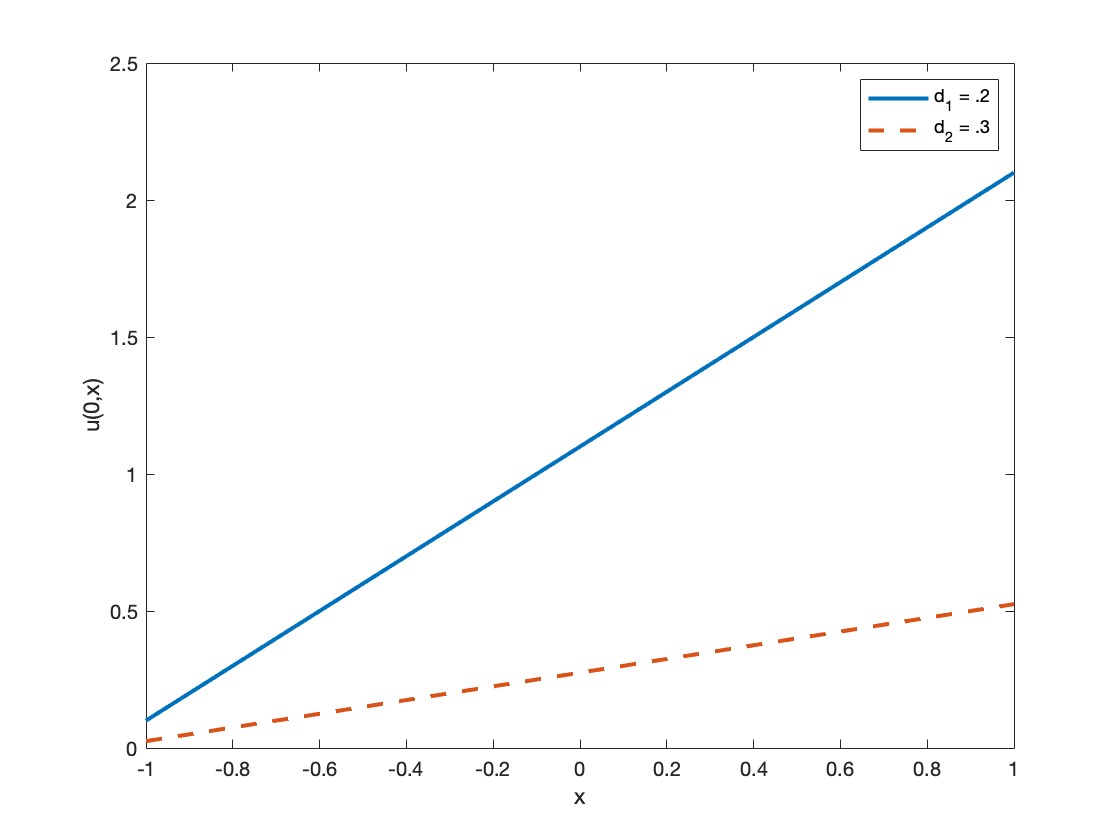}
        \caption{$p=2$, $t = 0$}
        \label{fig:classic-a}
    \end{subfigure}%
    ~ 
    \begin{subfigure}[t]{0.3\textwidth}
        \centering
        \includegraphics[height=1.2in]{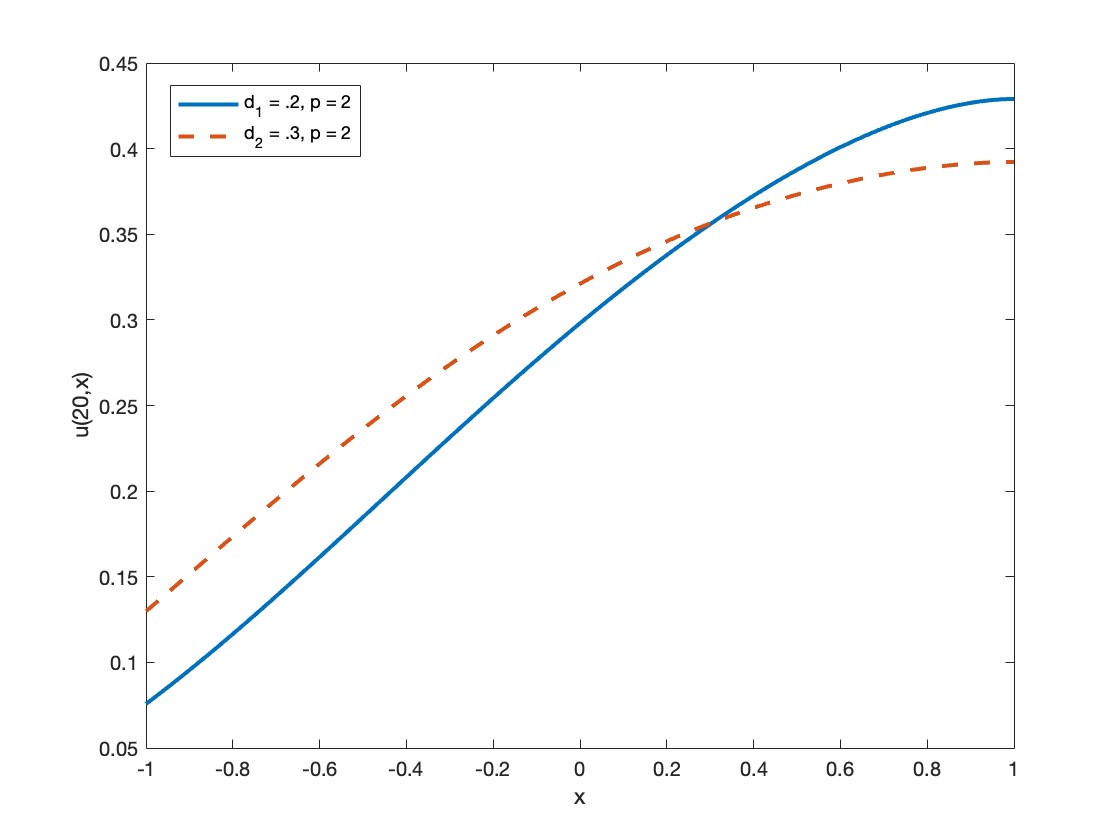}
        \caption{$p=2$, $t = 10$}
        \label{fig:classic-b}
    \end{subfigure}
        ~ 
    \begin{subfigure}[t]{0.3\textwidth}
        \centering
        \includegraphics[height=1.2in]{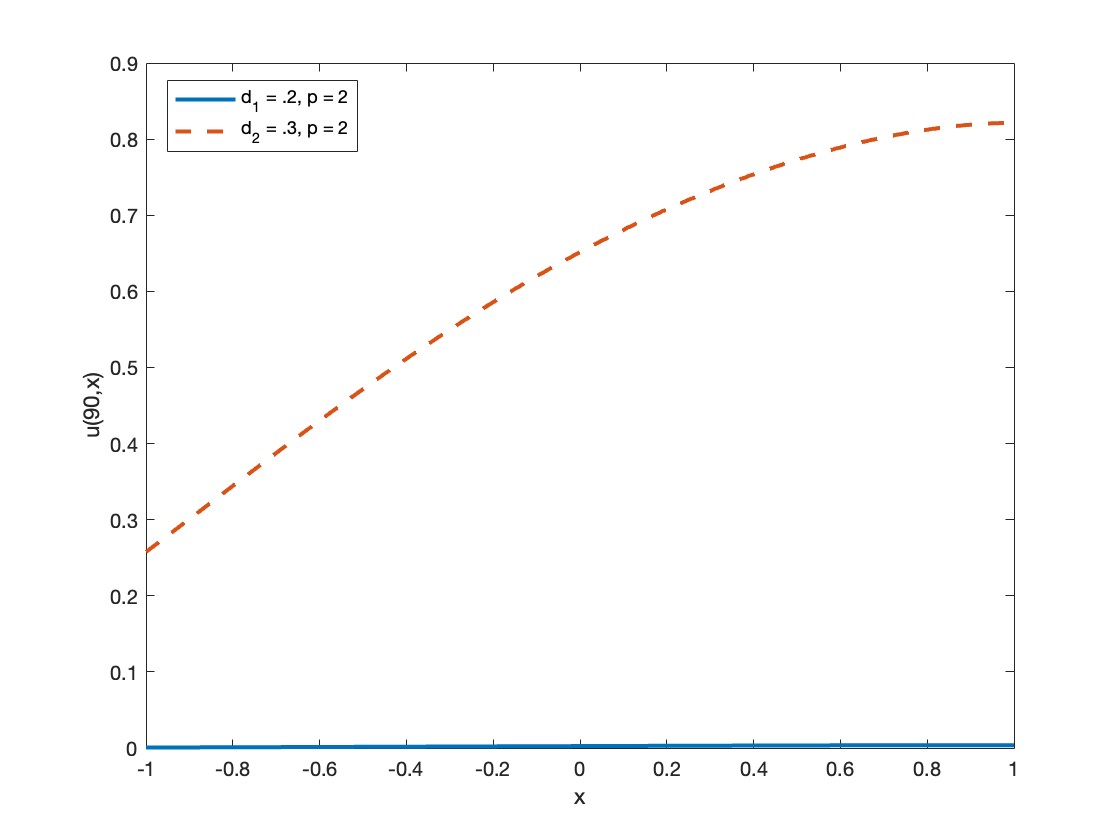}
        \caption{$p=2$, $t = 90$}
        \label{fig:classic-c}
    \end{subfigure}\\
        \begin{subfigure}[t]{0.3\textwidth}
        \centering
        \includegraphics[height=1.2in]{q-pt5-IC.jpg}
        \caption{$p=7/4$, $t = 0$}
        \label{fig:theorem1-d}
    \end{subfigure}%
    ~ 
    \begin{subfigure}[t]{0.3\textwidth}
        \centering
        \includegraphics[height=1.2in]{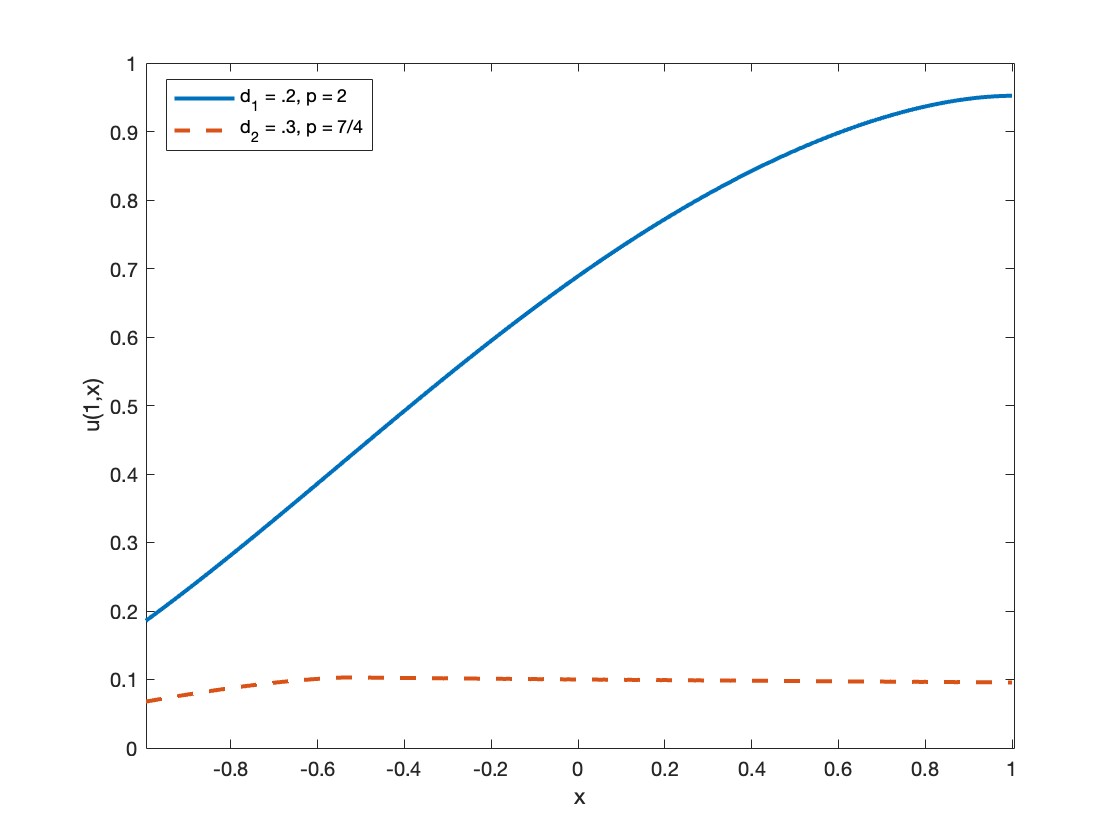}
        \caption{$p=7/4$, $t = 1$}
        \label{fig:theorem1-e}
    \end{subfigure}
        ~ 
    \begin{subfigure}[t]{0.3\textwidth}
        \centering
        \includegraphics[height=1.2in]{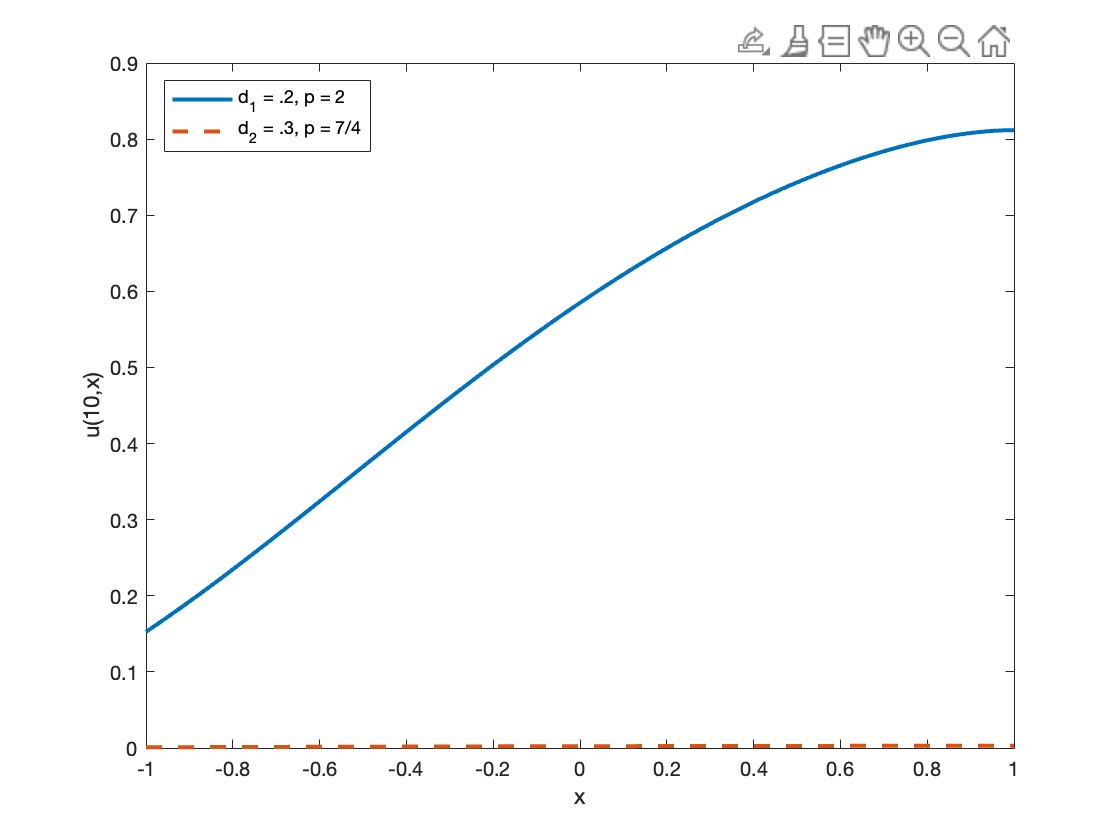}
        \caption{$p=7/4$, $t = 10$}
        \label{fig:theorem1-f}
    \end{subfigure}\\
     \begin{subfigure}[t]{0.3\textwidth}
        \centering
        \includegraphics[height=1.2in]{q-pt5-IC.jpg}
        \caption{$p=7/5$, $t = 0$}
        \label{fig:theorem2-g}
    \end{subfigure}%
    ~ 
    \begin{subfigure}[t]{0.3\textwidth}
        \centering
        \includegraphics[height=1.2in]{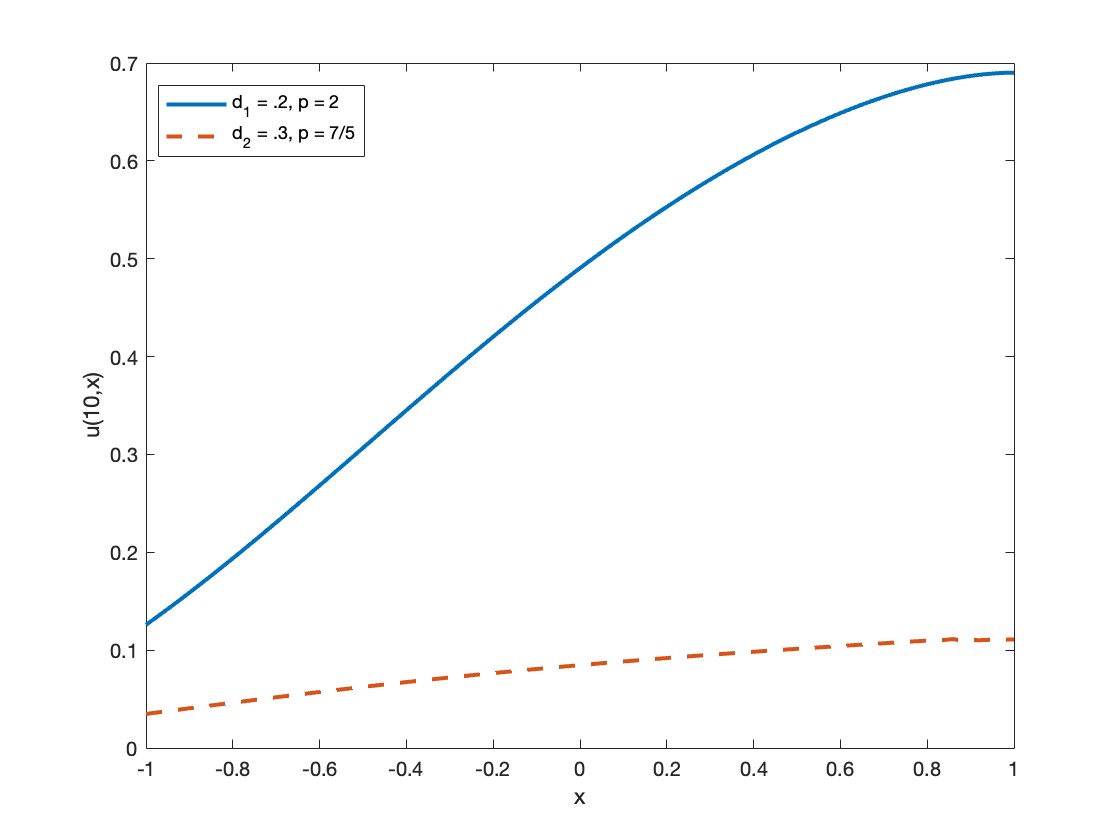}
        \caption{$p=7/5$, $t = 10$}
        \label{fig:theorem2-h}
    \end{subfigure}
        ~ 
    \begin{subfigure}[t]{0.3\textwidth}
        \centering
        \includegraphics[height=1.2in]{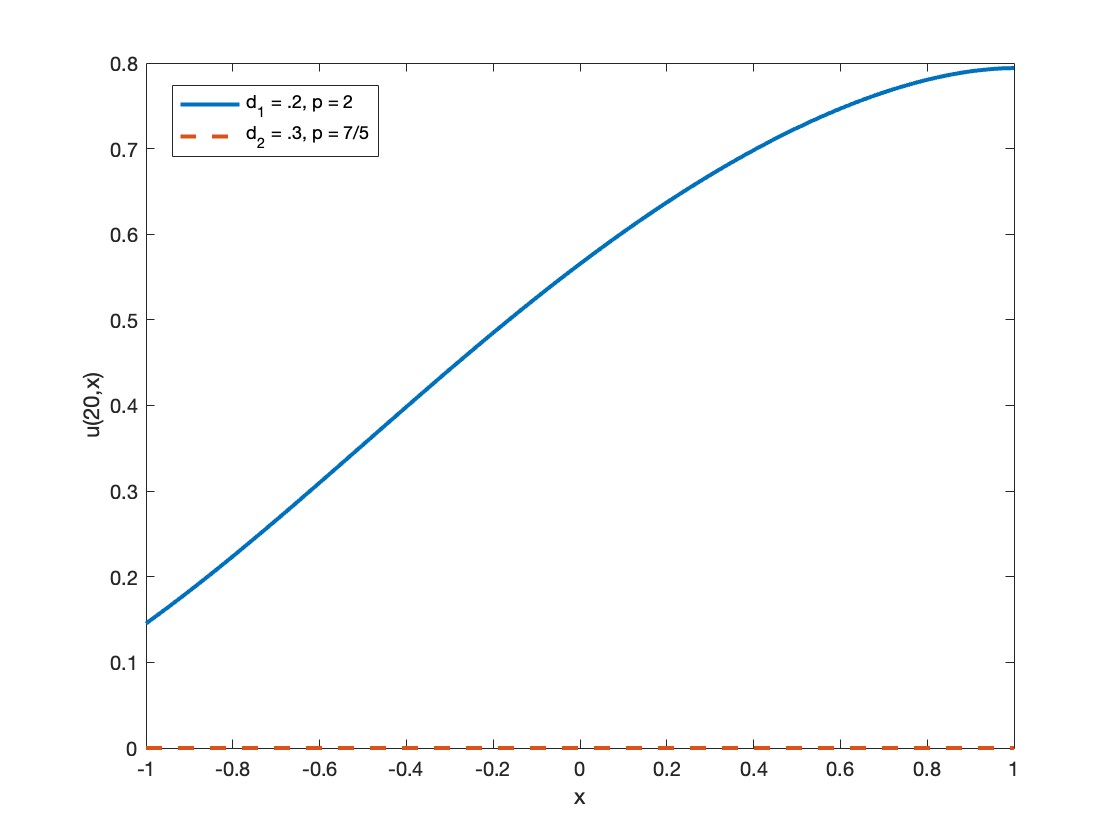}
        \caption{$p=7/5$, $t = 20$}
        \label{fig:theorem2-i}
    \end{subfigure}
    \caption{Solutions to system \eqref{eq:pde_modeld2} with $q=0.5$, $d_1 = 0.2$, and $d_2 = 0.3$.}
    \label{fig:theorem}
\end{figure}

In Figure \ref{fig:theorem}, we demonstrate the results of Theorem \ref{thm:FFTEdd} and \ref{thm:FFTEdd1} by comparing three cases: the classical case when $p=2$, the case where $\frac{3}{2}<p<2$, and the case where $\frac{4}{3}<p<\frac{3}{2}$. In each of these cases, we show the numerical solution at three different points in time: the initial distribution, an intermediate time point, and at equilibrium. In Figures \ref{fig:classic-a}-\ref{fig:classic-c}, both species follow classic linear diffusion ($p=2$). It is known that the fast diffuser competitively excludes the slow diffuser which we observe at equilibrium in Figure \ref{fig:classic-c}. Figures \ref{fig:theorem1-d}- \ref{fig:theorem1-f}, we compare the case where species $u$ follows linear diffusion, $p=2$, and species $v$ follows fast $p$-Laplacian diffusion with $p=7/4$, as in Theorem \ref{thm:FFTEdd}. With the same initial conditions as the classical case, if species $v$ follows non-linear diffusion with $\frac{3}{2}<p<2$, Figure \ref{fig:theorem1-f} illustrates the equilibrium where the slow diffuser prevails, contrary to when $p=2$. Finally, Figures \ref{fig:theorem2-g}-\ref{fig:theorem2-i} illustrate the results of Theorem \ref{thm:FFTEdd1}, where $\frac{4}{3} < p < \frac{3}{2}$. With the same initial condition as Figure \ref{fig:classic-a},  in the long run, the slow diffuser outcompetes the fast diffuser, shown in Figure \ref{fig:theorem2-i}.

\begin{figure}[h!]
    \centering
        \begin{subfigure}[t]{0.3\textwidth}
        \centering
        \includegraphics[height=1.2in]{q-pt5-IC.jpg}
        \caption{ $t = 0$}
        \label{fig:fraction-a}
    \end{subfigure}%
    ~ 
    \begin{subfigure}[t]{0.3\textwidth}
        \centering
        \includegraphics[height=1.2in]{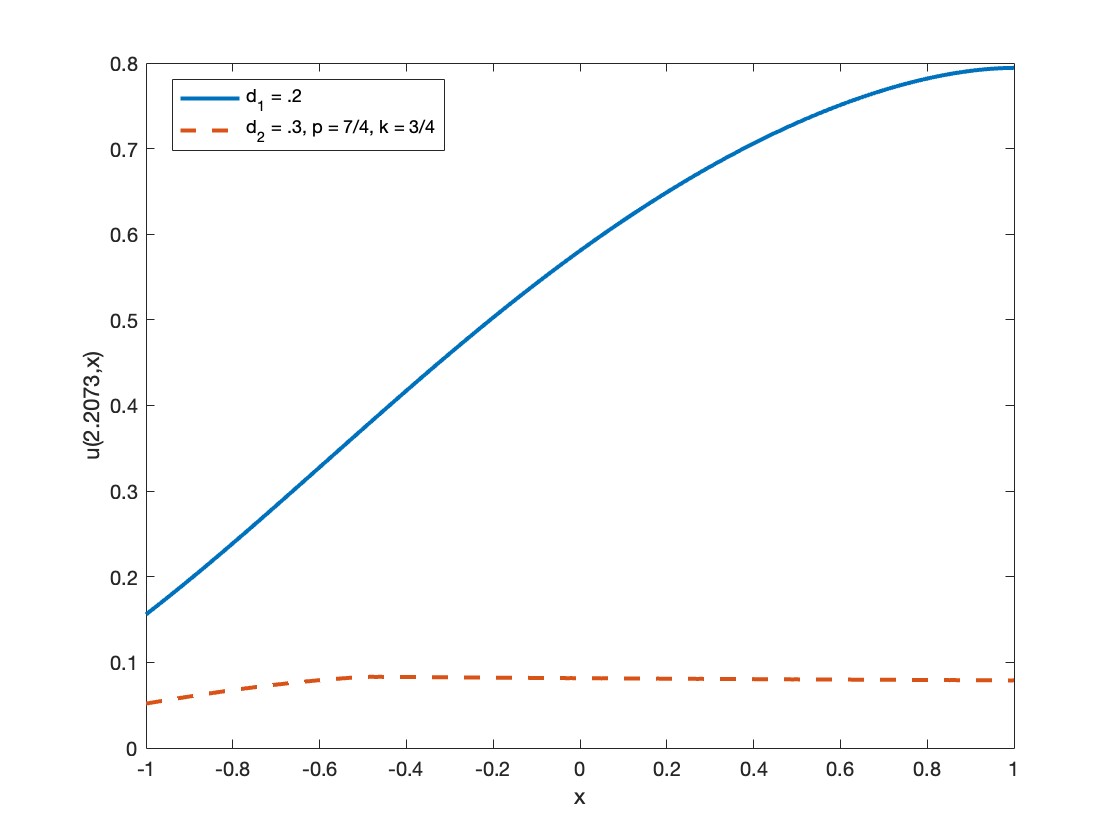}
        \caption{ $t = 1$}
        \label{fig:fraction-b}
    \end{subfigure}
        ~ 
    \begin{subfigure}[t]{0.3\textwidth}
        \centering
        \includegraphics[height=1.2in]{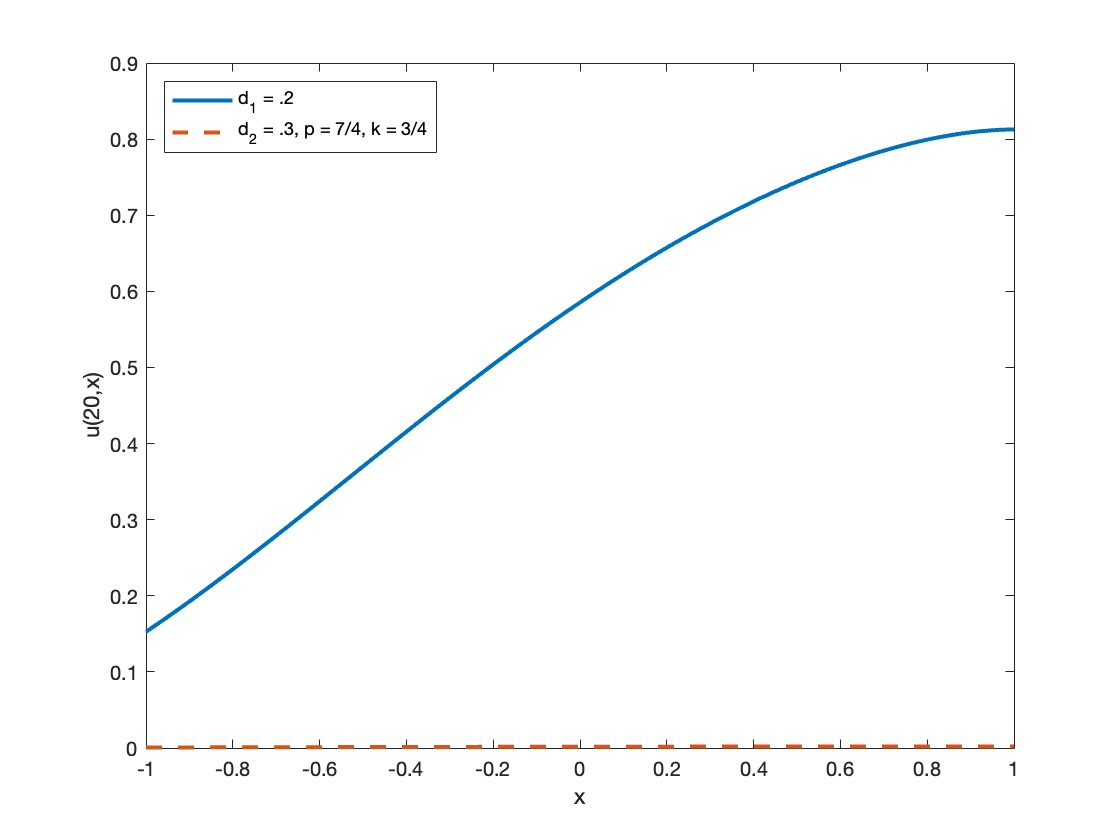}
        \caption{ $t = 10$}
        \label{fig:fraction-c}
    \end{subfigure}
    \caption{Solutions to system \eqref{eq:pde_modeld2-k} with and $p=7/4$, $q=0.5$, $d_1 = 0.2$, and $d_2 = 0.3$, and $k=3/4$.}
    \label{fig:fraction}
\end{figure}

It is also of interest to investigate the case where only a fraction of the population of species $v$ is using the $p$-Laplacian diffusion strategy and the rest of the population is diffusing with classical, linear diffusion as in equation \eqref{eq:pde_modeld2-k}. In Figure \ref{fig:fraction}, we illustrate the solutions to equation \eqref{eq:pde_modeld2-k} with $p = 7/4$ and $k=3/4$ at three different time points. In this simulation, three-quarters of the population $v$ is subject to $p$ -Laplacian diffusion and one-quarter is diffusing with $p=2$. Even when only part of the population is diffusing according to $p$-Laplacian diffusion, we can still observe a case where the slow diffuser prevails contrary to the classical case ($p=2$), shown in Figure \ref{fig:fraction-c}.

\begin{figure}[h!]
    \centering
        \begin{subfigure}[t]{0.31\textwidth}
        \centering
        \includegraphics[width = \textwidth]{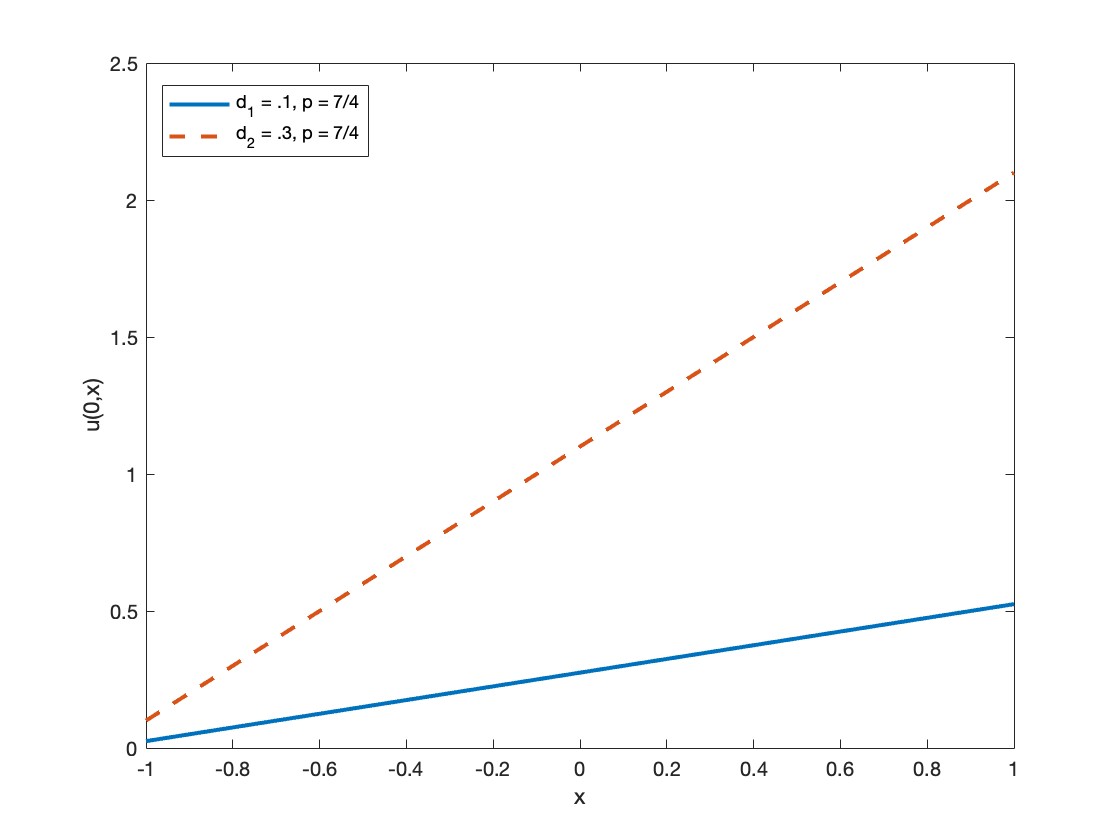}
        \caption{ $t = 0$}
        \label{fig:fastfast-drift-coefficients-a}
    \end{subfigure}%
    ~ 
    \begin{subfigure}[t]{0.31\textwidth}
        \centering
        \includegraphics[width = \textwidth]{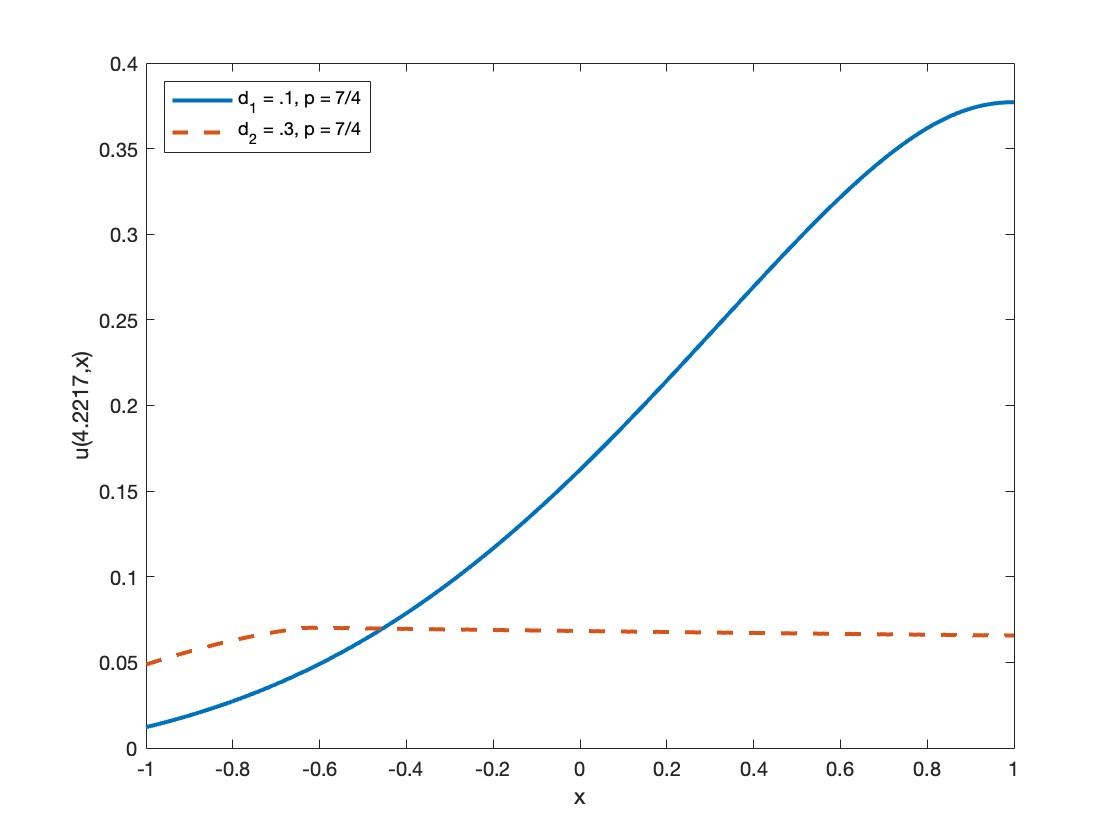}
        \caption{$t=4.2$}
        \label{fig:fastfast-drift-coefficients-b}
    \end{subfigure}%
    ~
    \begin{subfigure}[t]{0.31\textwidth}
        \centering
        \includegraphics[width = \textwidth]{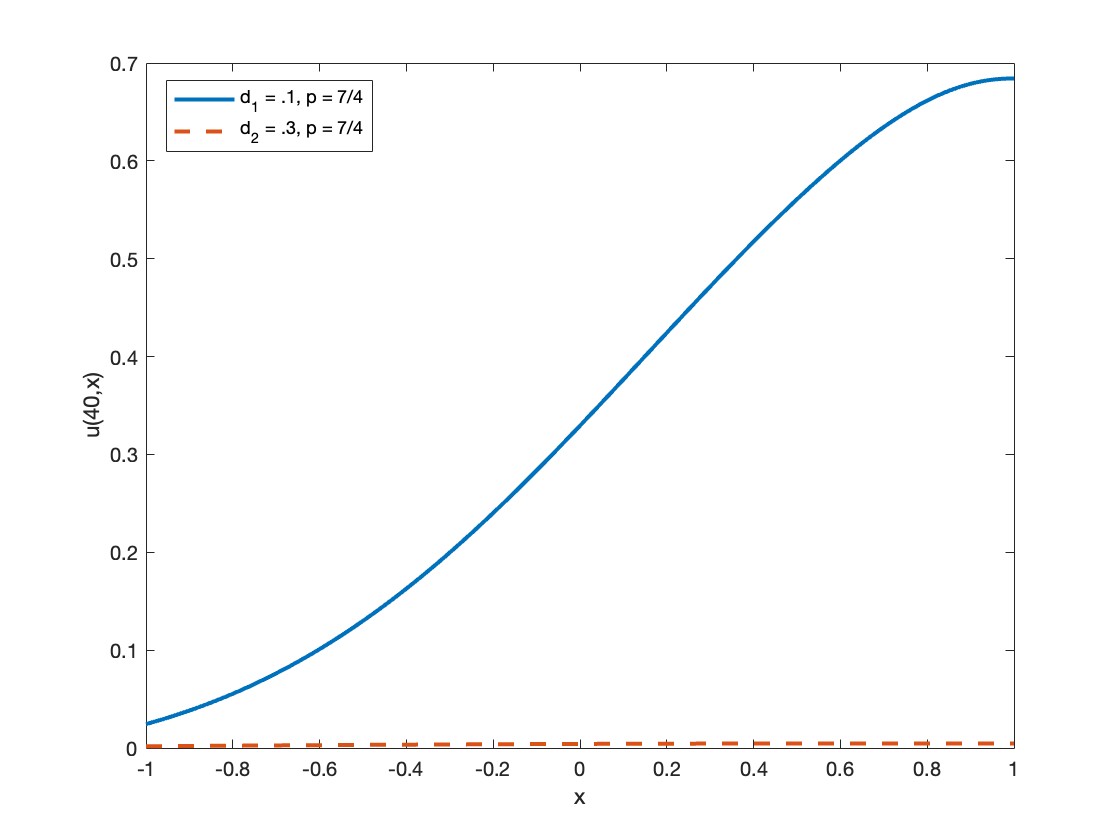}
        \caption{ $t = 40$}
        \label{fig:fastfast-drift-coefficients-c}
    \end{subfigure}
    \caption{Solutions to equation \eqref{eq:pde_modeld2}, with two species following $p-$Laplacian diffusion, $q=0.5$, $p=7/4$}
    \label{fig:fastfast-drift-coefficients}
\end{figure}

\begin{figure}[h!]
    \centering
        \begin{subfigure}[t]{0.31\textwidth}
        \centering
        \includegraphics[width = \textwidth]{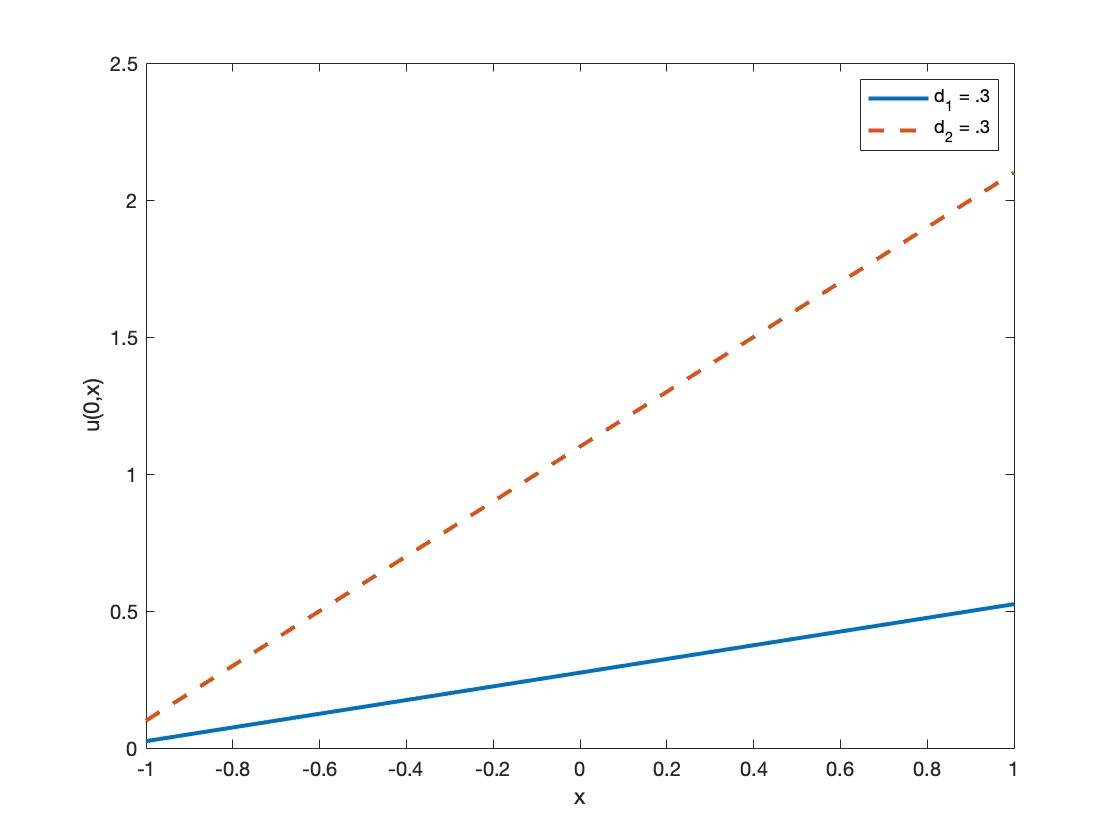}
        \caption{ $t = 0$}
        \label{fig:fastfast-drift-p-a}
    \end{subfigure}%
    ~ 
    \begin{subfigure}[t]{0.31\textwidth}
        \centering
        \includegraphics[width = \textwidth]{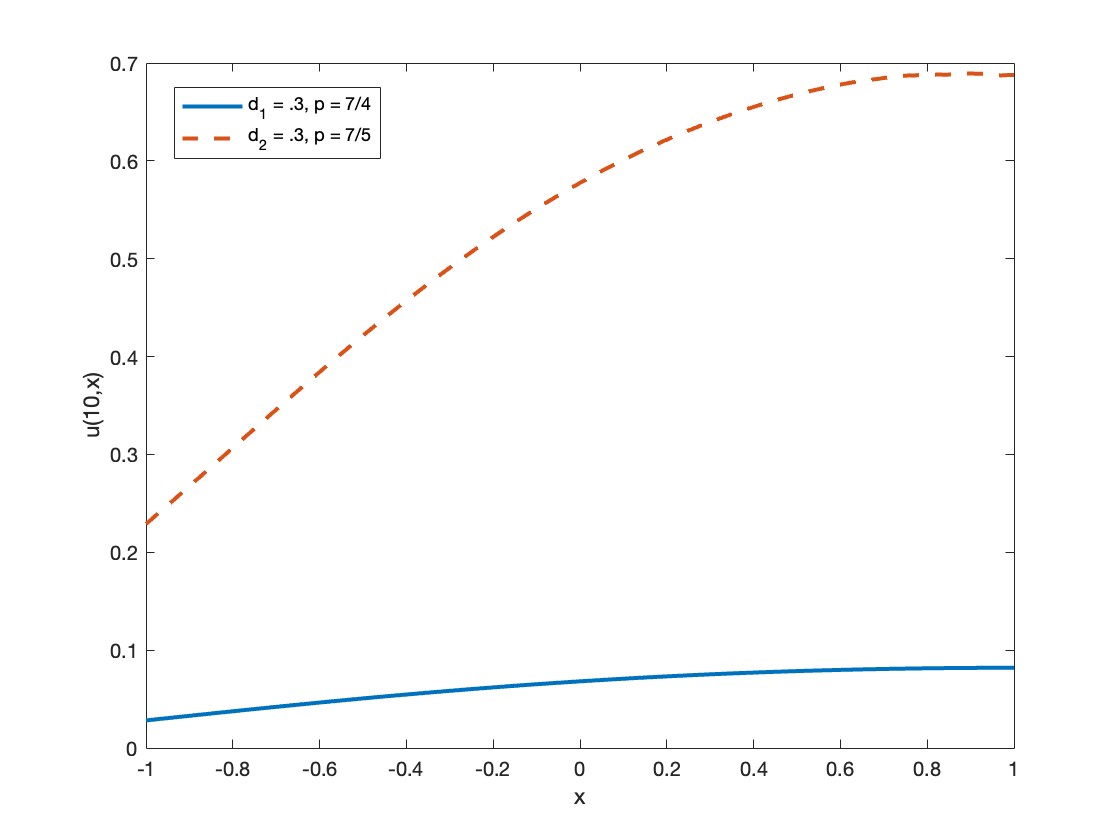}
        \caption{$t=10$}
        \label{fig:fastfast-drift-p-b}
    \end{subfigure}%
    ~
    \begin{subfigure}[t]{0.31\textwidth}
        \centering
        \includegraphics[width = \textwidth]{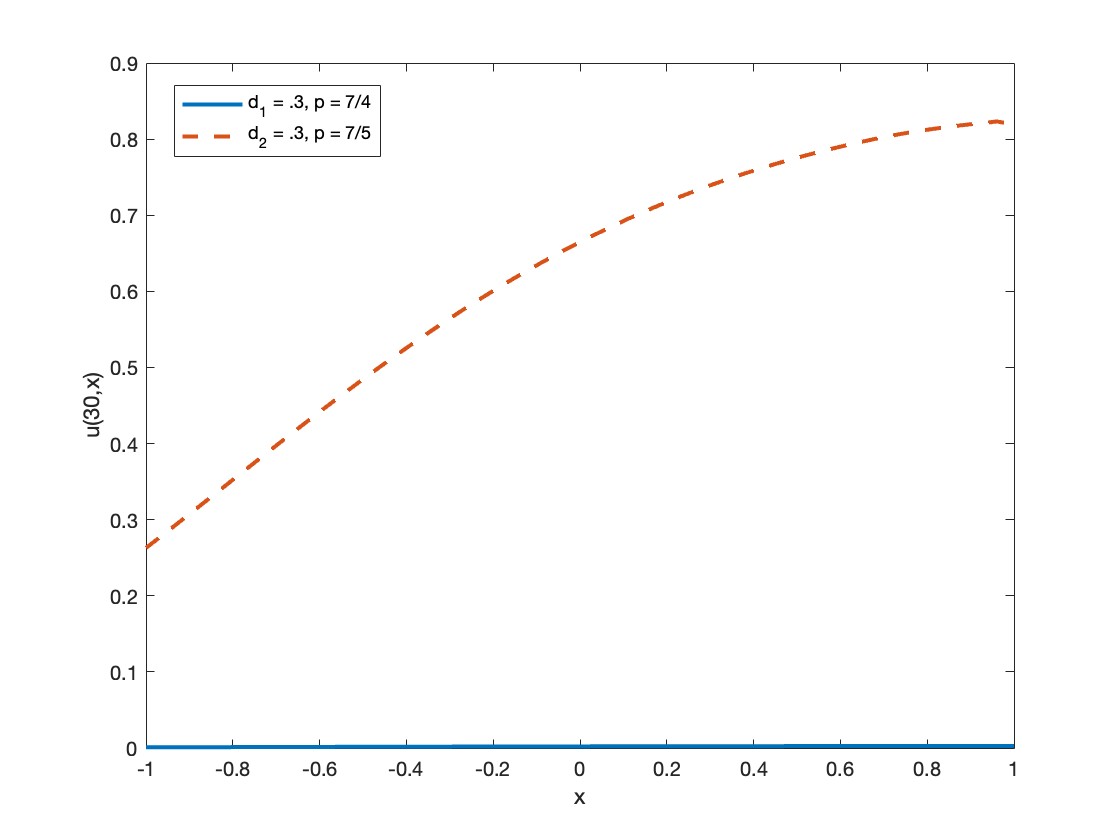}
        \caption{ $t = 30$}
        \label{fig:fastfast-drift-p-c}
    \end{subfigure}\\
        \begin{subfigure}[t]{0.31\textwidth}
        \centering
        \includegraphics[width = \textwidth]{two-p-1pt75-p-1pt4-q-pt5-t0.jpg}
        \caption{ $t = 0$}
        \label{fig:fastfast-drift-p-d}
    \end{subfigure}%
    ~ 
    \begin{subfigure}[t]{0.31\textwidth}
        \centering
        \includegraphics[width = \textwidth]{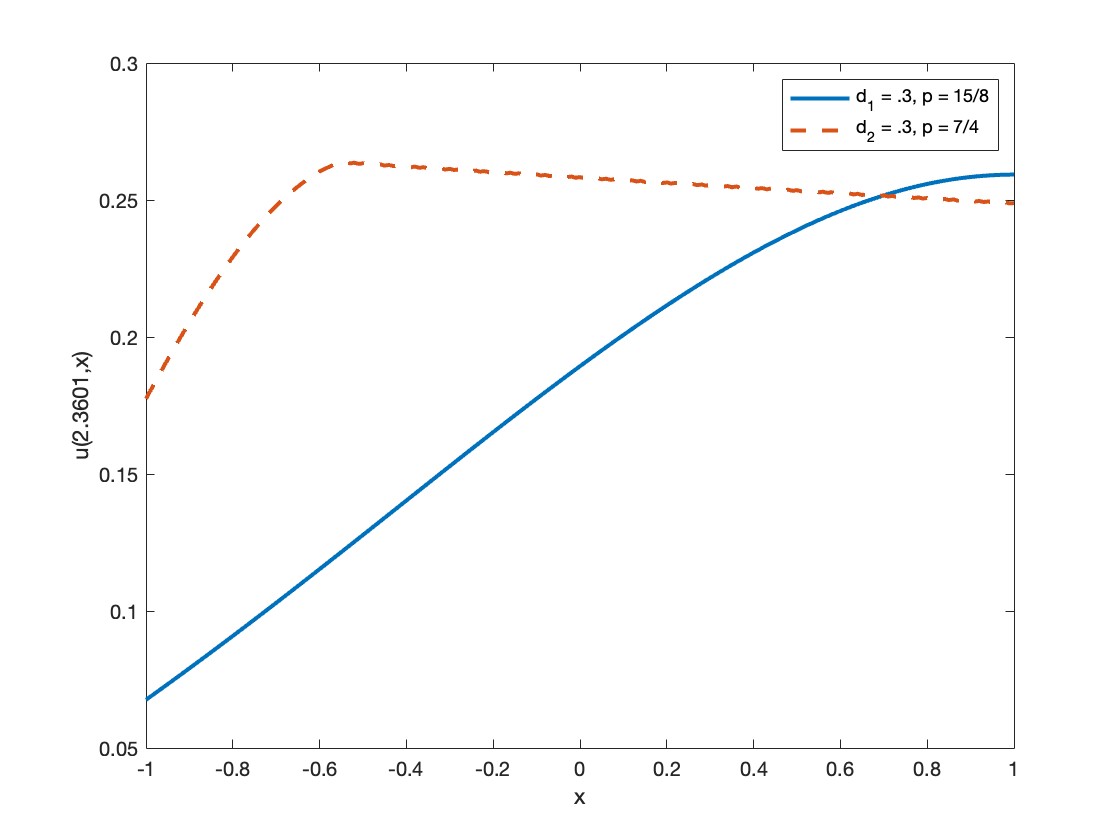}
        \caption{$t=2$}
        \label{fig:fastfast-drift-p-e}
    \end{subfigure}%
    ~
    \begin{subfigure}[t]{0.31\textwidth}
        \centering
        \includegraphics[width = \textwidth]{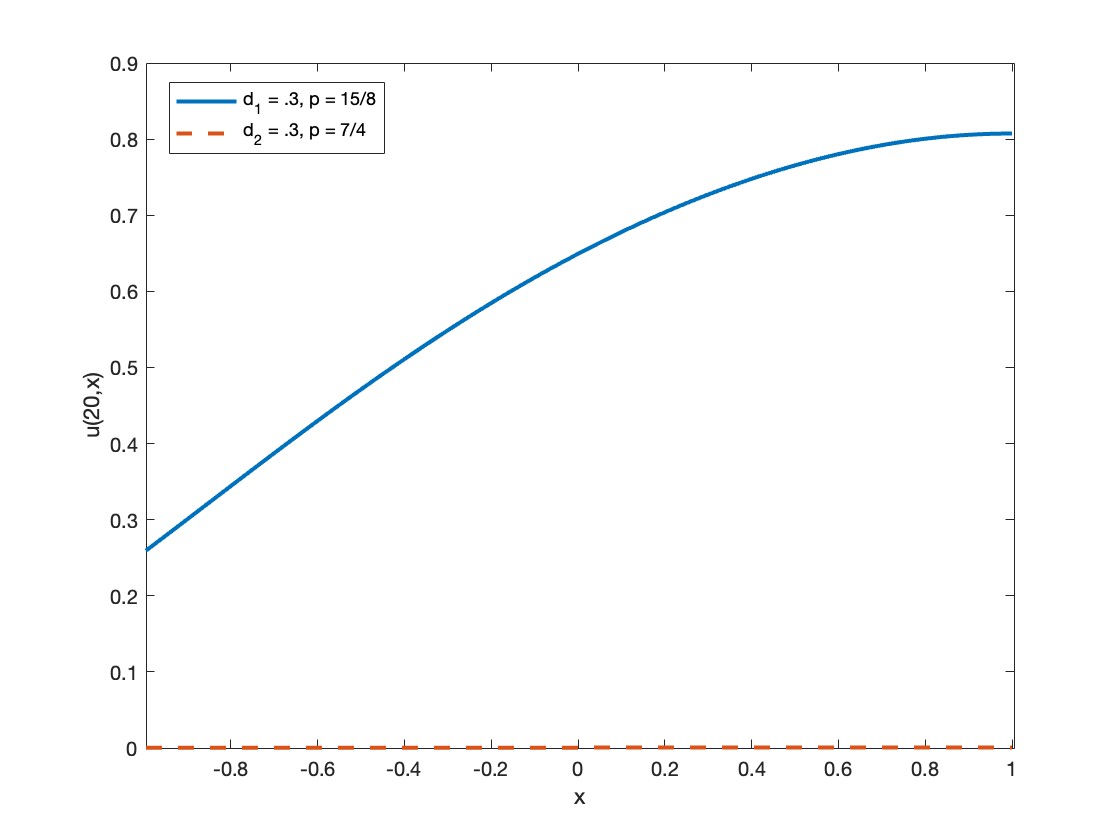}
        \caption{ $t = 20$}
        \label{fig:fastfast-drift-p-f}
    \end{subfigure}
    \caption{Solutions to equation \eqref{eq:pde_modeld2}, with two species following $p-$Laplacian diffusion, $q=0.5$, $d_1 = d_2 = 0.3$.}
    \label{fig:fastfast-drift-p}
\end{figure}

Although we do not investigate cases with both species following $p$-Laplacian diffusion analytically, it is interesting to investigate these cases numerically to motivate future work. In Figure \ref{fig:fastfast-drift-coefficients}, both species are subject to $p$-Laplacian diffusion with $p=7/4$. Despite initially having a smaller population density, shown in Figure \ref{fig:fastfast-drift-coefficients-a}, the slow diffuser outcompetes the fast diffuser, Figure \ref{fig:fastfast-drift-coefficients-c}. This is contrary to the classical case when $p=2$ and the fast diffuser always outcompetes the slow diffuser. We also consider cases where $d_1 = d_2$, but the two species have different $p$-values. We present two such cases in Figure \ref{fig:fastfast-drift-p}. In the first case, Figures \ref{fig:fastfast-drift-p-a}-\ref{fig:fastfast-drift-p-c}, the species with the smaller $p$-value outcompetes the species with the larger $p$-value, shown in Figure \ref{fig:fastfast-drift-p-c}. In the second case, the opposite happens. The species with the larger $p$-value outcompetes the species with the smaller $p$-value \ref{fig:fastfast-drift-p-f}. Ultimately, these simulations suggest there are rich dynamics that can occur when both species are following $p$-Laplacian diffusion with $p \in (1,2]$, and there is still much to be explored.

\section{Discussion and Conclusion}

In the current manuscript, we considered a novel model for two-species competition, given by \eqref{eq:pde_modeld211}. Here, one of the competitors disperses via the $p$-Laplacian operator, when $1 < p < 2$. This is the case of ``fast" diffusion. Equation \eqref{eq:pde_modeld211} is a degenerate equation, and it needs to be analyzed in the weak framework. Our main result is to show global in time existence of weak solutions to \eqref{eq:pde_modeld211} for any positive initial data in the regime $\frac{3}{2} < p < 2$, via Theorem \ref{thm:mt1}. The dynamics of \eqref{eq:pde_modeld211} are also very interesting, in particular while the faster diffuser always wins in the drift case, via classical results ($p=2$ case), we see the ``much faster" diffuser could lose ($1<p<2$ case) - this is seen via Theorems \ref{thm:FFTEdd}-\ref{thm:FFTEdd1}.

The numerical investigation of equation \eqref{eq:pde_modeld211} illustrates specific examples of Theorem \ref{thm:FFTEdd} and Theorem \ref{thm:FFTEdd1}. These investigations highlight when a species diffuses with a $p$-Laplacian diffusion strategy, classical results may not hold, and the slow diffuser may outcompete the fast diffuser in an environment with drift. Additionally, these investigations highlight the dynamics still to be explored with this system, particularly if both species are using a $p$-Laplacian diffusion strategy.

\bibliographystyle{unsrt}
\bibliography{ref}

\appendix

\section{Functional Preliminaries}

\begin{lemma}
\label{lem:wlp}
Consider exponents $0 < p_0 < p_1 < \infty$ and a domain $\Omega$ that is a closed and bounded in $\mathbb{R}^{n}$, $n\geq 1$, and a $u(x) \in L^{p_1}(\Omega)$. Then,

\begin{equation}
||u||_{L^{p_{\theta}}(\Omega)} \leq C||u||^{1-\theta}_{L^{p_{0}}(\Omega)} ||u||^{\theta}_{L^{p_{1}}(\Omega)}
\end{equation}

for all $0 \leq \theta \leq 1$, and where $p_{\theta}$ is defined as $\frac{1}{p_{\theta}} = \frac{1-\theta}{p_0} + \frac{\theta}{p_1}$.
\end{lemma}

\begin{lemma}
\label{lem:gns}
Consider a $\phi \in W^{m,q^{'}}(\Omega) \cap L^{q}(\Omega)$. Then if $p^{'}, q^{'}, q \geq 1, 0 \leq \theta \leq 1$, and

\begin{equation}
k - \frac{n}{p^{'}} \leq \theta \left(  m - \frac{n}{q^{'}} \right) - (1 - \theta) \frac{n}{q},
\end{equation}
there exists a constant C such that,

\begin{equation}
	||\phi||_{W^{k,p^{'}}(\Omega)} \leq C ||\phi||^{\theta}_{W^{m,q^{'}}(\Omega)} ||\phi||^{1 - \theta}_{L^{q}(\Omega)} 
	\end{equation}

\end{lemma}	

Now we state the classical Aubin-Lions compactness Lemma,

\begin{lemma}
\label{lem:al}
Let $X_0, X$ and $X_1$ be three reflexive Banach spaces with 

$X_0 \hookrightarrow \hookrightarrow X \hookrightarrow X_1$. That is $ X_0$ is compactly embedded in $X$ and that $X$ is continuously embedded in $X_1$. For $1\leq p,q \leq \infty$, let

 \begin{equation}
W = \{ u \in L^{p^{'}}([0,T];X_0) \ | \  u^{'} \in L^{q^{'}}([0,T];X_1)  \}
\end{equation}

(i) If $p^{'} < \infty$ then the embedding of $W$ into $L^{p^{'}}([0,T];X)$ is compact.

(ii) If $p^{'} = \infty$ and $q^{'} > 1$ then the embedding of $W$ into $C([0,T];X)$ is compact.
\end{lemma}

\begin{lemma}
\label{lem:pc1}
    Consider functions $\phi, \sigma \in C^{1}(\Omega), \Omega \subset \mathbb{R}^{n}, n=1,2, |\Omega| < \infty$, $\epsilon > 0$. Then if $p \geq 2$ we have,
    \begin{equation}
        (|\nabla \phi|^{2} + \epsilon)^{\frac{p-2}{2}}|\nabla \phi|^{2} \geq |\nabla \phi|^{p},
\end{equation}

\begin{equation}
        (|\nabla \phi|^{2} + \epsilon)^{\frac{p-2}{2}}|\nabla \phi| \leq C\left(|\nabla \phi|^{p}+1\right).
\end{equation}

and 

\begin{equation}
        (|\phi|^{p-2}\phi - |\sigma|^{p-2}\sigma)\cdot (\phi - \sigma) \geq C |\phi - \sigma|^{p}
\end{equation}

\end{lemma}

We next present the main result of this section.

\begin{theorem}
Consider the system \eqref{eq:pde_model}. Then there exists initial data $(u_{0},v_{0}) \in W^{1,2} (\Omega)$ for which there exist local weak solutions to \eqref{eq:pde_model}.
\end{theorem}

\end{document}